\documentclass[11pt]{article}
\usepackage{graphicx} 
\usepackage{amsmath} \usepackage{amssymb} \usepackage{latexsym} \usepackage{enumitem} \usepackage{mathrsfs} \usepackage{comment}
\usepackage{color}

\usepackage[colorlinks=true,urlcolor=blue,
citecolor=red,linkcolor=blue,linktocpage,pdfpagelabels,bookmarksnumbered,bookmarksopen]{hyperref}

\newtheorem{assumption}{Assumption}

\def\qed{ \ \vrule width.2cm height.2cm depth0cm\smallskip}
\newenvironment{proof}{\noindent {\bf Proof.\/}}{$\qed$\vskip 0.1in}

\newcommand{\ba}{\begin{array}}
\newcommand{\ea}{\end{array}}
\newcommand{\be}{\begin{equation}}
\newcommand{\ee}{\end{equation}}
\newcommand{\bea}{\begin{eqnarray}}
\newcommand{\eea}{\end{eqnarray}}
\newcommand{\beaa}{\begin{eqnarray*}}
\newcommand{\eeaa}{\end{eqnarray*}}

\def\dbE{\mathbb{E}}
\def\dbF{\mathbb{F}}

\def\dbI{\mathbb{I}}

\def\dbL{\mathbb{L}}

\def\dbP{\mathbb{P}}
\def\dbR{\mathbb{R}}
\def\dbS{\mathbb{S}}

\def\dbQ{\mathbb{Q}}
\def\dbZ{\mathbb{Z}}

\def\a{\alpha}
\def\b{\beta}
\def\g{\gamma}
\def\d{\delta}
\def\e{\varepsilon}
\def\z{\zeta}
\def\l{\lambda}

\def\n{\nu}
\def\si{\sigma}

\def\th{\theta}
\def\o{\omega}

\def\G{\Gamma}
\def\D{\Delta}

\def\O{\Omega}
\def\cA{{\cal A}}
\def\cB{{\cal B}}
\def\cC{{\cal C}}

\def\cE{{\cal E}}
\def\cF{{\cal F}}

\def\cI{{\cal I}}

\def\cL{{\cal L}}
\def\cM{{\cal M}}

\def\cP{{\cal P}}

\def\cW{{\cal W}}
\def\cX{{\cal X}}

\def\no{\noindent}

\def\ms{\medskip}

\def\q{\quad}

\def\pa{\partial}

\def\qed{ \hfill \vrule width.25cm height.25cm depth0cm\smallskip}

\newcommand{\basa}{\begin{assumption}}
\newcommand{\easa}{\end{assumption}}

\newcommand{\bas}{\begin{assum}}
\newcommand{\eas}{\end{assum}}

\def\pa{\partial}

\def\tl{\tilde}

\def\1{{\bf 1}}

\def\:{\!:\!}

at 9pt

\definecolor{alp}{rgb}{0.0, 0.5, 0.0}

\newtheorem{thm}{Theorem}[section]
\newtheorem{lem}[thm]{Lemma}

\newtheorem{rem}[thm]{Remark}
\newtheorem{eg}[thm]{Example}
\newtheorem{defn}[thm]{Definition}
\newtheorem{assum}[thm]{Assumption}

\begin{document}

\title{\bf  Heterogeneous Mean Field Games and Local Well-posedness} 
\author{Bixing Qiao\thanks{\noindent  Dept. of Math., 
University of Southern California. E-mail:
\href{mailto:bqiao@usc.edu}{bqiao@usc.edu}}  
}
\date{}
\maketitle

\begin{abstract} 
Motivated by the recent interests in asymmetric mean field games, this paper provides a general framework of Heterogeneous Mean Field Game (HMFG) that subsumes different formulations of graphon mean field games. The key feature of the HMFG is that the players interact with the population through the density ensemble. In this case, the HMFG system becomes an infinite-dimensional Forward-Backward SDE (FBSDE) system. We show that the FBSDE is locally well-posed, thus the HMFG has a unique equilibrium. In addition, we show that the equilibrium of HMFG is a good approximate equilibrium of the corresponding N-Player Game. Lastly, we derive the It\^{o} formula of infinite-dimensional measure flow and use it to obtain the master equation for HMFG as a decoupling field of the infinite-dimensional FBSDE system.
\end{abstract}

\no{\bf Keywords:}  Heterogeneous Mean Field Game, master equation, propagation of chaos

\ms
\no{\it 2020 AMS Mathematics Subject Classification:}  91A16, 91A43, 35R15, 49N80, 93E20

\vfill\eject


\label{sect-Introduction}

\section{Introduction}
Proposed independently by Huang-Malhame-Caines \cite{huang2006large} and Lasry-Lions \cite{Lasry-Lions}, Mean Field Games (MFG) theory has become an important tool in economics and finance. The c lassical MFG theory considers the games with infinitely many homogeneous players, where each agent is influenced by the population only through its empirical distribution. The symmetric structure yields a single representative system describing all players. Such simple structure attracts many models to adopt the MFG theory for describing systems with large populations. We refer to the book by Achdou-Cardaliaguet-Delarue-Porretta-Santambrogio \cite{ACDPS} and Carmona-Delarue \cite{carmona2018probabilistic} for the general theory of MFG.  In addition, for applications of MFG, see e.g. Achdou-Han-Lasry-Lions-Moll \cite{AHLLM}, Aurell-Carmona-Dayanikli-Lauri\`{e}re \cite{Aurell2022, Aurell2022b}, Cardaliaguet-Lehalle \cite{CardaliaguetLehalle}, Carmona-Cormier-Soner \cite{Carmona02092023}, Carmona-Dayanıklı-Lauri{\`e}re \cite{Carmona2022}, Carmona-Delarue-Lacker \cite{Carmona2017},  Carmona-Fouque-Sun \cite{SystemRisk}, Carmona-Fouque-Mousavi-Sun \cite{Carmona2018}, Carmona-Graves \cite{Carmona2020}, Élie-Ichiba-Laurière \cite{elie2020}, Ichiba-Yang \cite{yang2024}, Lacker-Zariphopoulou \cite{LackerZariphopoulou}, Lacker-Soret \cite{LackerSoret}, Tangpi-Wang \cite{Tangpi2024,Tangpi2025}, and Tangpi-Zhou \cite{tangpi2023optimal}. 


While having been very successful, the MFG theory has a drawback. The strong symmetric structure it requires can be hard to meet in applications. A prominent example is systemic-risk modeling, where standard MFG formulations assume that banks contribute equally to correlations and affect all others identically, see for example \cite{ Carmona2017, Carmona2018, SystemRisk,elie2020}. In practice, interbank lending/borrowing is heterogeneous. Furthermore, empirical studies indicate that an asymmetric network structure is crucial, see e.g., \cite{huser2015too}. This motivate researchers to propose a model that can incorporate an asymmetric structure while retaining MFG tractability.

A popular asymmetric extension of the MFG is the Graphon MFG (GMFG) model, introduced by Aurell-Carmona-Lauriere \cite{aurell2021stochastic}, Caines-Huang \cite{caines2019graphon} and Carmona-Cooney-Graves-Lauriére \cite{GMFG1}. Graphon is a symmetric, measurable function: $w:[0,1]^2\mapsto\dbR$ that represents the limit of some dense network structure with vertices go to infinity. The value of $w(\a,\b)$ encodes the connection strength between node $\a$ and $\b$. In GMFG, the players interact with the population through some kind of aggregation terms over the graphon. Under this idea, different formulations are proposed. In \cite{aurell2021stochastic}, the authors use the aggregation of the states of all players over the graphon. In \cite{caines2019graphon} and \cite{Crucianelli24}, the authors use the aggregation of the laws of all players over the graphon. In \cite{Pham2024}, the authors consider a normalized aggregation of the laws of all players over the graphon.

In this project, we propose a unifying framework for MFG with asymmetric agents, which we call the Heterogeneous Mean Field Game (HMFG). We consider a system with a continuum of players, each interacting with the population through the law ensemble of all types. We observe that the graphon aggregation of the law of all players as well as the graphon weighted average of the law of all players are functions of the law ensemble of all players. Furthermore, by the so-called exact law of large numbers, the graphon aggregation of the states of all players is in fact a function of population law ensemble as well. Consequently, the models in \cite{aurell2021stochastic, caines2019graphon, Pham2024,Crucianelli24} arise as special cases of HMFG.

A key analytical challenge in the HMFG is working with a continuum of players. In particular, the model requires a continuum of Brownian motions that are pairwise independent while also integrable with respect to the index space in order to define the aggregate term. In \cite{aurell2021stochastic}, the Fubini extension space is introduced to tackle the problem. For the general theory of the Fubini extension space, see \cite{SUN200631,SUN2009432}. In \cite{Pham2024}, the authors discuss the space of population law ensemble $\cM$ and define graphon as an operator from $\cM$ to itself. In this paper, we further discuss the space $\cM$ under weaker topology and show that it is a complete metric space. This is crucial for a contraction mapping argument establishing local well-posedness.

Having set the infinite-dimensional space of state processes, we define the equilibrium of the HMFG as a fixed point. Given population law ensemble $\rho^*$, the player $\th$ finds optimal control $\a^{*,\th}$; when these strategies are used, the resulting law ensemble is again $\rho^*$. We take the population law ensemble $\rho^*$ as the equilibrium of the HMFG, since different optimal controls could result in the same $\rho^*$. This definition resembles that of the MFG model. Next we derive the coupled infinite-dimensional PDE system. Each player solves a Hamilton-Jacobi-Bellman (HJB) equation for the optimazation problem and a Fokker-Planck equation for the law of her dynamic. The PDE systems of all players are coupled together through the solution ensemble of the Fokker-Planck equations. We also derive the corresponding FBSDE characterization. The FBSDE viewpoint is convenient because the process ensemble naturally lives in the Fubini extension; measurability in the index space then follows directly. Proving local well-posedness of the coupled infinite-dimensional FBSDE yields existence and uniqueness of the HMFG equilibrium.

A second objective of this project is to establish the connection between the $N-$player game and its HMFG limit, commonly referred to as propagation of chaos, see e.g. \cite{carmona2018probabilistic}. Under conditions that guarantees the truncated model are close to the HMFG model, we show that the HMFG equilibrium is an approximate equilibrium for the corresponding N-player game. This parallels known GMFG results obtained under step-graphon convergence and continuity of aggregate terms, see e.g. \cite{aurell2021stochastic, caines2019graphon, cao2025}. 

Finally, HMFG framework enables a dynamic analysis via the master equation, provided the game value is unique. Following the MFG literature, the master equation is the PDE that characterizes the dynamic value function of the game. In the HMFG model, the master equation becomes an infinite-dimensional system of equations. For each player, the dynamic game value function is a function of time, state, and the marginal population law ensemble that satisfies a PDE. To derive the master equation, we first derive the It\^o formula for functions defined along the flow of the population law ensemble. Apply the It\^o formula, we show that the master equation of HMFG is a decouping field of the FBSDE system and thus it provides a unique solution to the HMFG. The well-posedness of master equation of HMFG is a serious topic and will be further studied in our future project.
The master equation of MFG is first introduced by Lions in \cite{Lions08}. Researchers make serious efforts on its global well-posedness, see e.g. Bertucci \cite{Bertucci02122023}, Bertucci-Cecchin \cite{Bertucci2024},  Bertucci-Lasry-Lions \cite{Bertucci04032019}, Cardaliaguet-Delarue-Lasry-Lions \cite{cardaliaguet2019master}, Cardaliaguet-Souganidis \cite{Cardaliaguet2022}, Chassagneux-Crisan-Delarue\cite{CCD}, Gangbo-M{\'e}sz{\'a}ros \cite{Gangbo2022}, Gangbo-Meszaros-Mou-Zhang \cite{GMMZ}, Gangbo-Świech \cite{GANGBO20156573}, Graber-M{\'e}sz{\'a}ros \cite{Graber2024}, Mou-Zhang \cite{mou2019wellposedness, MZ,MZ2}, Mou-Zhang-Zhou \cite{mou2025}.


Here we discuss more on previous research of GMFG. In the work of Lacker-Soret \cite{LS}, authors propose to consider GMFG as a special type of MFG with the index as an additional dimension. One challenge in GMFG is how to model the limit of game problems on sparse networks. Since graphon can only be attained through taking the limit of a dense graph. Several works have attempted to address this issue. In \cite{LRW}, the authors address the issue by considering local weak convergence of marked graphs. In \cite{FFI, FFI2} and \cite{Crucianelli24}, the authors address the issue by exploiting the special structure of certain types of networks. 
Another stream of literature works on numerically solving the GMFG system. In the work by Cui-Koeppl \cite{Kai}, numerical solution of GMFG has been explored using reinforcement learning. In the work of Aurell-Carmona-Dayanıklı-Laurière \cite{Aurell2022}, the authors consider a GMFG on finite state space, thus reducing the infinite-dimensional system to a finite-dimensional one.


The rest of the paper is organized as follows. In Section \ref{sect-setup}, we introduce the notations, Fubini Extensions, probability measure ensembles and related technical tools. In Section \ref{sect-HMFG}, we state the heterogeneous $N-$player game and the corresponding HMFG. We define the equilibrium of HMFG and state the PDE system and the FBSDE system characterizations. In Section \ref{sect-smallT}, we establish the local well-posedness of the FBSDE and, thus, the existence of the unique equilibrium of the HMFG. In Section \ref{sect-convergence}, we show that the equilibrium of the HMFG is a desired approximate equilibrium the corresponding $N-$player game. In Section \ref{sect-MasterEq}, we derive the master equation of HMFG as the decoupling field of the FBSDE system.  
\section{Setup}\label{sect-setup}
\subsection{Fubini Extension}

Consider a filtered probability space $(\Omega,\cF,\dbP)$.
Denote $(I,\cB_I,\l_I)$ as the index space, where $I=[0,1]$, and $\cB_I$ is the Borel $\si-$algebra on $I$ and $\l_I$ is the Lebesgue measure on $I$. For any Polish space $E$, $\cP(E)$ denotes the set of Borel probability measures on $E$.

We will study an inhomogeneous model, thus we need to consider infinite-dimensional random variable $\xi$ that is measurable with respect to both $\dbP$ and $\l_I$. We want $\forall \th,\tl\th\in I$, $\xi^\th$ is independent with $\xi^{\tl\th}$. In addition, we want $\xi$ to satisfy the following Fubini property: 
$$\int_{I\times\Omega}\xi^\th(\o)d\dbQ(\o,\th)=\int_I \dbE \xi^\th d\th=\dbE\int_I  \xi^\th d\th$$
The solution is to extend the product space $(I\times\O,\cB_I\times\cF,\l_I\times \dbP)$ to its Fubini Extension defined by Sun-Zhang in \cite{SUN2009432}. 
\begin{defn} Let $(I\times\O,\cB_I\times\cF,\l_I\times \dbP)$ be the usual product probability spaces of $(I,\cB_I,\l_I)$ and $(\O,\cF,P)$. A probability space $(I\times\O,\cW,\dbQ)$ extending $(I\times\O,\cB_I\times\cF,\l_I\times \dbP)$ is said to be a \textit{Fubini Extension} if for any $\dbQ$-integrable function $\xi$ on $(I\times \O)$,\\
    (1) $\xi^\cdot(\o)$ is $(I,\cB_I,\l_I)$-integrable for $\l_I$-almost all $\o\in\O$, and $\xi^\th(\cdot)$ is $(\O,\cF,\dbP)$-integrable for $\dbP$-almost all $\th\in I$.\\
    (2) $\dbE \xi^\th $ is integrable on $(I,\cB_I,\l_I)$ and $\int_I  \xi^\th(\o) d\th$ is integrable on $(\O,\cF,\dbP)$. Furthermore 
    $$\int_{I\times\Omega}\xi^\th(\o)\l_I(d\th)=\int_I \dbE \xi^\th \l_I(d\th)=\dbE\int_I  \xi^\th \l_I(d\th).$$
\end{defn}
In addition, we need the following theorem from \cite{SUN2009432} to construct a random variable $\xi$ that is measurable with respect to the Fubini extension space and $\xi^\th$ and $\xi^{\tl\th}$ are pairwise independent for $\th,\tl\th\in I$ and $\th\neq\tl\th$
\begin{thm}\label{fubiniEx}
    Let $X$ be a complete separable metric space. There exists a probability space $(I, \cB_I, \l_I)$ extending the Lebesgue unit interval, a probability space $(\O,\cF, \dbP)$, and a Fubini extension $(I\times\O, \cB_I\boxtimes\cF, \l_I\boxtimes\dbP)$ such that for any measurable mapping $\phi$ from $(I, \cB_I, \l_I)$ to the space $\cP(X)$ of Borel probability measures on $X$, there is a $\cB_I\boxtimes\cF$-measurable process $\xi:I\times\O\mapsto X$ such that the random variables $\xi_\th$ are essentially pairwise independent, and the distribution $\dbP\circ f_\th^{-1}$ is the given distribution $\xi_\th$ for $\l_I$-almost all $\th\in I$.
\end{thm}
By  essentially pairwise independent (e.p.i.), we mean that for $\l_I$-almost all $\th\in I$, $\xi^\th$ is independent with $\xi^{\tl\th}$ for $\l_I$-almost all $\tl \th\in I$.

Now $\forall\th\in I$, let $\phi(\th)$ be the law of the Brownian motion. Then by Theorem \ref{fubiniEx}, there exists an $\cF\boxtimes\cI-$measurable Brownian motion $B$ such that $B$ is e.p.i.. Denote $\dbF^\th$ as the filtration generated by $B^\th$.

One of the consequences of the Extended Fubini Theorem is the so-called Exact Law of Large Numbers (ELLN) by Sun \cite{SUN200631}: if we integrate over $I$ for the e.p.i. Brownian motion ensemble $B$ that is constructed above at any $t\in T$, we get
$$\int_I B_t^\th \l_I(d\th)=\dbE^{\dbP\boxtimes \l}B_t=\int_I\dbE B_t^\th \l_I(d\th)=0$$

Let us denote $L_\boxtimes^2(\O\times I;C([0,T];
\dbR^d))$ as the space of $\dbP\boxtimes\l$-square integrable random variables, i.e. for all $X\in L_\boxtimes^2(\O\times I;C([0,T];
\dbR^d))$:
$$\int_{\O\times I} \sup_{t\in[0,T]}|X_t^\th(\omega)|^2\dbP\boxtimes\l(d\o,d\th)<+\infty$$

\subsection{Topology on Measure Valued Functions}
Similar to the approach in MFG, we want to define the equilibrium of HMFG through the fixed point argument. In MFG model, the player depends on the population through a single law that represents the distribution of the population dynamic. Since HMFG considers the scenario of a large number of heterogeneous agents, each player will depend on the population through an probability distribution ensembles. In this section, we will study the topology of such objects.

For any Polish space $E$ equipped with $p-$metric $d_p$, we equip the space of probability measure $\cP(E)$ with Wasserstein p-distance $W_p$ metric for $p\geq1$: for any $\mu,\nu\in\cP(E)$ 
$$W_p(\mu,\nu)=\inf_{\rho\in\G(\mu,\nu)}(\dbE^\rho |d(X,Y)|^p)^\frac{1}{p}$$
where $\G(\mu,\n)$ is the set of joint distribution with marginal distribution $\mu$ and $\nu$ and $X\sim\mu$, $Y\sim\n$. 

When $E:=C([0,T],\dbR^d)$, $(E,\|\cdot\|_p)$ is complete and thus $(\cP(E),W_p)$ is also complete. Denote the space $\cM$ as the set of Lebesgue measurable functions $\phi:I\mapsto\cP(E)$. Let us define the function $d_p:\cM\times \cM \mapsto\dbR$: for $\phi,\psi\in\cM$
$$d_p(\phi,\psi):=\int_I W_p(\phi(\th),\psi(\th))\l_I(d\th).$$
We have the following result concerning the new metric $d_p$.
\begin{thm}\label{completeness}
    $(\cM, d_p)$ is a complete metric space.
\end{thm}
\begin{proof} First let us show that $d_p$ is a metric. It suffices to check the triangle inequality: for $\phi_1,\phi_2,\phi_3\in\cM$
    \begin{equation*}
        \begin{aligned}
            d_p(\phi_1,\phi_2)=\int_I W_p(\phi_1(\th),\phi_2(\th))\l_I(d\th)&\leq\int_I W_p(\phi_1(\th),\phi_3(\th))+W_p(\phi_3(\th),\phi_2(\th))\l_I(d\th)\\
            &=d_p(\phi_1,\phi_3)+d_p(\phi_3,\phi_2)
        \end{aligned}
    \end{equation*}

    Next, we need to show that $(\cM, d_p)$ is complete. Let $\{\phi_n\}_{n=1}^\infty\subset\cM$ be a Cauchy sequence in $(\cM, d_p)$, i.e. $\lim_{m,n\to0}d_p(\phi_n,\phi_m)=0$, we need to show that $\lim_{n\to0}d_p(\phi_n,\phi)=0$ for some $\phi\in\cM$. By Bounded Convergence Theorem (BCT), we have:
    $$0=\lim_{m,n\to\infty}d_p(\phi_n,\phi_m)=\lim_{m,n\to\infty}\int_I W_p(\phi_n(\th),\phi_m(\th))\l_I(d\th)=\int_I \lim_{m,n\to\infty}W_p(\phi_n(\th),\phi_m(\th))\l_I(d\th)$$
    Thus $\lim_{m,n\to\infty}d_p(\phi_n(\th),\phi_m(\th))=0$ $\l_I$-almost surely. Since $(\cP(E),W_p)$ is complete, we have $\forall \th\in I$, $\exists\phi_\th\in \cP(E)$ such that $\lim_{n\to\infty}W_p(\phi_\th,\phi_n(\th))=0$. We define function $\phi:I\mapsto\cP(E)$ by $\phi(\th)=\phi_\th$. If $\phi$ is $\l_I-$measurable, then by BCT again, we have
    $$\lim_{n\to\infty}d(\phi_n,\phi)=\lim_{n\to\infty}\int_I W_p(\phi(\th),\phi_n(\th))\l_I(d\th)=\int_I \lim_{n\to\infty}W_p(\phi(\th),\phi_n(\th))\l_I(d\th)=0.$$
    
    Remains to show that $\phi(\th)$ is measurable. Since the Borel $\si-$algebra $\cB(\cP(E))$ is generated by set
    $$\cE=\{G_1=(f^G)^{-1}(G_0):f^G(\rho)=\rho(G) \q \forall \rho\in\cP(E),G\in\cB(E),\q G_0\in \cB(\dbR)\}$$
    It suffices to check $\forall G\in\cB(E)$, $f^G\circ\phi$ is $\cB(I)-$measurable. Since $\phi_n(\cdot)$ is measurable and $f^G$ is continuous, then $f^G\circ\phi_n(\cdot)$ is $\cB(I)-$measurable. Furthermore, $\lim_{n\to\infty}f^G\circ\phi_n=f^G\circ\phi$ point-wise and thus $f^G\circ\phi(\cdot)$ is $\cB(I)-$measurable. Thus for all $G_1\in\cE$, $\phi^{-1}(G_1)=\phi^{-1}\circ(f^G)^{-1}(G_0)=(f^G\circ\phi)^{-1}$ is $\cB(I)-$measurable. Since $\si(\cE)=\cB(\cP(E))$, $\phi$ is $\cB(I)-$measurable.
\end{proof}
\begin{rem}
    For $X\in\dbL_\boxtimes^2(\O\times I;E)$, define function $\phi(\th)=\cL(X^\th)$. We will check that $\phi\in\cM$. Specifically, we want to know if $\phi$ is $\cB(I)-$measurable. By definition, we know the mapping $\phi:\th\mapsto \dbP\circ (X^\th)^{-1}(G)$ is $\cB(I)-$measurable for all $G\in\cB(E)$. By the same argument in the proof of theorem \ref{completeness}, the map $\th\mapsto \dbP\circ (X^\th)^{-1}$ is $\cB(I)-$measurable. 

Furthermore, since $X\in\dbL_\boxtimes^2(\O\times I;E)$, $X_t$ is also $\dbP\boxtimes\cB(I)-$measurable. Thus by the same argument the map $\phi_t:\th\mapsto \dbP\circ (X_t^\th)^{-1}$ is also $\cB(I)-$measurable. For the remainder of the paper, we will denote $\cL(X)=\phi$ and $\cL(X_t)=\phi_t$.
\end{rem}

Let us also define $\cM_0$ as the set of Lebesgue measurable functions $\phi_0:I\mapsto\cP(\dbR^d)$ and metric $d_p^0$ for $p\geq1$. For $\phi_0,\psi_0\in\cM_0$:
\begin{equation}
    \begin{aligned}
        d_p^0(\phi_0,\psi_0)&=\int_I W_2(\phi_0(\th),\psi_0(\th))\l_I(d\th)\\
        &=\int_I \inf_{\rho\in\G(\phi_0(\th),\psi_0(\th))}\big[\dbE^\rho\big(|X_0^\th-Y_0^\th|\big )^p \big]^{1/p}\l_I(d\th)
    \end{aligned}
\end{equation}
where $\G(\phi_0(\th),\psi_0(\th))$ denotes the set of joint probability measure with marginals $\phi_0(\th)$ and $\psi_0(\th)$. The completeness of $(\cM_0,d_p^0)$ is the same as the proof of Theorem \ref{completeness}, but our result doesn't rely on such a result, thus we will not state the theorem separately. However, we will show the following useful lemma.
\begin{lem}
For $\phi,\bar\phi\in\cM$, $t\in[0,T]$ and $p\geq1$, we have: $$d_p(\phi_t,\bar\phi_t)\leq d_p(\phi,\bar\phi) $$
\end{lem}
\begin{proof}
\begin{equation*}
\begin{aligned}
    d_p(\phi_t,\bar\phi_t)=&\int_I \inf_{\rho\in\G(\phi_t(\th),\bar\phi_t(\th))}\big[\dbE^\rho\big(|X_t^\th-Y_t^\th|\big )^p \big]^{1/p}\l_I(d\th)\\
    =&\int_I \inf_{\rho\in\G(\phi(\th),\bar\phi(\th))}\big[\dbE^\rho\big(|X_t^\th-Y_t^\th|\big )^p \big]^{1/p}\l_I(d\th)\\
    \leq&\int_I \inf_{\rho\in\G(\phi(\th),\bar\phi(\th))}\big[\dbE^\rho\sup_{t\in[0,T]}\big(|X_t^\th-Y_t^\th|\big )^p \big]^{1/p}\l_I(d\th)=d_p(\phi,\bar\phi)
\end{aligned}
\end{equation*}
 \end{proof}

\begin{rem}
    We introduce the space $\cM_0$ for two reasons. First, the equilibrium value $V$ of the HMFG, which we will define later, depends on the initial time, the initial state of the player and the initial population law ensemble $\mu\in \cM_0$. Second, we will define the utility and dynamics as a function of the population marginal law ensemble $\phi_t\in\cM_0$ at time $t\in[0,T]$. 
\end{rem}

Our model concerns the following infinite-dimensional McKean–Vlasov type SDE system:
\begin{equation}\label{sde}
    dX_t^\th=b(\th,t,X_t^\th,\cL(X_t))dt+\si(\th,t,X_t^\th,\cL(X_t))dB_t^\th \q\forall\th\in I
\end{equation}
When $b,\si$ are uniformly Lipschitz continuous in $(x,\mu)$ for all $(\th,t)$, $X$ has a unique solution in $L_\boxtimes^2(\O\times I;C([0,T]; \dbR^d))$. The proof is similar with standard arguments of the well-posedness of SDE, thus we will not show here.

\subsection{Linear Functional Derivative}
Let $X\in L_\boxtimes^2(\O\times I;C([0,T],R))$ be the solution of the infinite-dimensional SDE system \eqref{sde}, we want to find the SDE for infinite-dimensional process $Y_t:=V(\th,t,X_t^\th,\cL(X_t))$. Thus, we need to define a derivative on $\cM_0$ and an It\^o formula for functional of marginal population law ensemble $\cL(X_t)\in\cM_0$. For derivative, we mimic the linear function derivative of function on $\cP(\dbR^d)$.

\begin{defn}
Consider a function $f: \cM_0\mapsto\dbR$. The linear functional derivative of $f$ is a function $\frac{\d f}{\d \mu}: \cM_0\times I\times \dbR\longrightarrow\dbR$ such that for $\mu,\nu\in\cM_0$ 
$$f(\nu)-f(\mu) = \int_0^1\int_I\int_{\dbR}\frac{\d f}{\d \mu}(t\nu+(1-t)\mu,x,\th)(\nu^\th-\mu^\th)(dx)\l_I(d\th)(dt)$$
\end{defn}

\begin{eg}
Define smooth function $\psi:\dbR^d\to\dbR$ and $h:\dbR\to\dbR$, $g(\mu^\th) = \int_{\dbR}\psi(x)\mu^\th(dx)$. Consider $f(\mu) = h(\int_0^1 g(\mu^\th)b(\th)\l_I(d\th))$ then we have:
\begin{equation*}
    \begin{aligned}
        &f(\nu)-f(\mu) \\
        =& \int_0^1 h'\Big(t\int_I g(\nu^{\tl\th})b(\tl\th)\l_I(d\tl\th)+(1-t)\int_I g(\mu^{\tl\th})b(\tl\th)\l_I(d\tl\th)\Big)\Big(\int_I (g(\nu^\th)-g(\mu^\th))b(\th)\l_I(d\th)\Big)dt\\
        =&\int_0^1 h'\Big(\int_I b(\tl\th)\int_\dbR\psi(x)(t\nu^{\tl\th}+(1-t)\mu^{\tl\th})(dx)\l_I(d\tl\th)\Big)\Big(\int_I(b(\th) \int_\dbR \psi(x)(\nu^\th-\mu^\th)(dx))\l_I(d\th)\Big)dt\\
        =&\int_0^1\int_I \int_\dbR h'\Big(\int_0^1 g(t\nu^{\tl\th}+(1-t)\mu^{\tl\th})b(\tl\th)\l_I(d\tl\th)\Big)\psi(x)b(\th)(\nu^\th-\mu^\th)(dx)\l_I(d\th) dt
    \end{aligned}
\end{equation*}
Thus $$\frac{\d f}{\d \mu}(\mu,x,\th) = h'\big(\int_Ig(\mu^{\tl\th})b(\tl\th)\l_I(d{\tl\th})\big)\psi(x)b(\th).$$
Notice that we can think of $f(\mu) = h(\int_I g(\mu^\th)b(\th)d\th)=h(\dbE^{\dbP\boxtimes\l}\phi(X(\o,\th))b(\th))$ where $X(\o,\th)\in L_\boxtimes^2(\O\times I)$ and $\mu^\th=P\circ X^{-1}(\cdot,\th)$. Than the result resemble the derivative rule for $\mu \in\mathcal{P}_2(\dbR^d)$: for $f=h(\dbE[f(X)])$ where $X \sim \mu$, $\frac{\d f}{\d \mu}=h'(\dbE[f(X)])f(x)$.
\end{eg}

\subsection{It\^o formula}
Now we will derive the It\^o Formula for functions of processes in $L_\boxtimes^2(\O\times I;C^1([0,T],R))$.
\begin{thm}\label{ito} Assume function $f:\cM_0\longrightarrow\dbR$ has a linear functional derivative $\frac{\d f}{\d \mu}(\mu,x,\th)$. Let us further assume that $\frac{\d f}{\d \mu}$ has continuous first and second derivative in $x$: $\partial_x\frac{\d f}{\d \mu},\partial_{xx}\frac{\d f}{\d \mu}$. Suppose both derivatives are Lipschitz continuous in $\mu$ with Lipschitz constant $L$. Define infinite-dimensional processes $b,\si\in L_\boxtimes^2(\Omega\boxtimes I)$ then the following SDE has a solution $X\in L_\boxtimes^2(\Omega\boxtimes I;C^1([0,T],R))$:
$$dX_t^\th=b_t^\th dt+\sigma_t^\th dW_t^\th$$
where $\mu_t^\th = \mathcal{L}({X_t^\th})$, we have:
\begin{equation}
        f(\mu_T)-f(\mu_0)  = \int_{I}\Big[\int_0^T\dbE\Big(\partial_x\frac{\d f}{\d \mu}(\mu_t,X_t^{\th},\th) b_t^{\th}+\frac{1}{2}\dbE\partial_{xx}\frac{\d f}{\d \mu}(\mu_{t},X_t^{\th},\th)\Big)|\si_t^{\th}|^2 dt\Big]d\th 
\end{equation}
\end{thm}
\begin{proof}
Let $0=t_0<t_1,......<t_N=T$ 
\begin{equation*}
    \begin{aligned}
        &f(\mu_T)-f(\mu_0)\\ 
        =& \sum_{i=0}^{N-1}f(\mu_{t_{i+1}})-f(\mu_{t_i})\\
        =&\sum_{i=0}^{N-1} \int_0^1\int_{I}\Big(\dbE\frac{\d f}{\d \mu}(\l\mu_{t_{i+1}}+(1-\l)\mu_{t_i},X_{t_{i+1}}^{\th},\th)-\dbE\frac{\d f}{\d \mu}(\l\mu_{t_{i+1}}+(1-\l)\mu_{t_i},X_{t_{i}}^{\th},\th)\Big)\l_I(d\th) d\l\\
    \end{aligned}
\end{equation*}
Now apply the standard It\^o formula:
\begin{equation}
    \begin{aligned}
    f(\mu_T)-f(\mu_0)
        =&\sum_{i=0}^{N-1} \int_0^1\int_{I}\Big(\Big[\int_{t_i}^{t_{i+1}}\dbE\partial_x\frac{\d f}{\d \mu}(\l\mu_{t_{i+1}}+(1-\l)\mu_{t_i},X_t^{\th},\th)b_t^{\th}dt\\
        &+\frac{1}{2}\int_{t_i}^{t_{i+1}}\dbE\partial_{xx}\frac{\d f}{\d \mu}(\l\mu_{t_{i+1}}+(1-\l)\mu_{t_i},X_t^{\th},\th)|\si_t^{\th}|^2 dt\Big]\Big)\l_I(d\th) d\l\\
        =& \int_{{I}}\Big[\sum_{i=0}^{N-1}\int_{t_i}^{t_{i+1}}\Big(\int_0^1\dbE\partial_x\frac{\d f}{\d \mu}(\l\mu_{t_{i+1}}+(1-\l)\mu_{t_i},X_t^{\th},\th)d\l\Big) b_t^{\th}\\
        &+\frac{1}{2}\Big(\int_0^1\dbE\partial_{xx}\frac{\d f}{\d \mu}(\l\mu_{t_{i+1}}+(1-\l)\mu_{t_i},X_t^{\th},\th)d\l\Big)|\si_t^{\th}|^2 dt\Big]\l_I(d\th)\\
        :=&I_1^N+I_2^N
    \end{aligned}
\end{equation}
Then we take $N$ to infinity. Here we only show the convergence of $I_1^N$ as $I_2^N$ follows the same way:  
    \begin{equation}
        \begin{aligned}
            &\lim_{N\to\infty}\|I_1^N-\int_I\dbE^{\mu_t^{\th}}\int_0^T\partial_x\frac{\d f}{\d \mu}(\mu_{t},X_t^{\th},\th)b_t^{\th}dtd\th\|_2^2\\
            =&\lim_{N\to\infty}\int_I\dbE^{\mu_t^{\th}}\sum_{i=0}^{N-1}\int_{t_i}^{t_{i+1}}\Big|\Big(\int_0^1\partial_x\frac{\d f}{\d \mu}(\l\mu_{t_{i+1}}+(1-\l)\mu_{t_i},X_t^{\th},\th)d\l\Big)  - \partial_x\frac{\d f}{\d \mu}(\mu_{t},X_t^{\th},\th)b_t^{\th}\Big|^2 dt\l_I(d\th)\\
            \leq& L\lim_{N\to\infty}\sum_{i=0}^{N-1}\int_{t_i}^{t_{i+1}}\Big(\int_0^1 W_2(\l\mu_{t_{i+1}}+(1-\l)\mu_{t_i}, \mu_{t}) d\l\Big) |b_t^{\th}|^2dt\l_I(d\th)\\
            \leq& L\lim_{N\to\infty}\sum_{i=0}^{N-1}\int_{t_i}^{t_{i+1}}\Big(\int_0^1\int_{I} \dbE|\l(X_{t_{i+1}}^{\bar\th}+(1-\l)X_{t_{i}}^{\bar\th}-X_{t}^{\bar\th}|^2d\bar\th d\l\Big) |b_t^{\th}|^2dt\l_I(d\th)\\
        \end{aligned}
    \end{equation}
    The first inequality holds by the Lipschitz continuity of $\pa_x\frac{\d f}{\d\mu}$. The last equality holds by the definition of $W_2$ norm. Notice that
    $X\in L_\boxtimes^2(\Omega\boxtimes\mathcal{I};C^1([0,T],R))$, thus we can apply DCT the conclude that \eqref{ito} converge to 0. 
\end{proof}

\section{Heterogeneous Mean Field Games}\label{sect-HMFG}
\subsection{Heterogeneous N-Player Game}
In this section, we will formulate the N-player game model that motivates HMFG. Let us consider a model with $N$ players. The players are distributed over a graph $G_K$ with $K$ nodes. Let $\cC_l$ denote the index set of players on node $l$. and $N=\sum_{l=1}^K|\cC_l|$. We assume that the players on the same node are homogeneous and players on the different nodes are heterogeneous. We call players on node $i$ as type $i$ player. 

Denote $A$ as the action space. Denote $\cA^N$ as the set of admissible controls for $N-$player game and $\cA$ as the set of admissible controls for the infinite-player game which we will define in the next subsection. In this paper, we consider closed-loop admissible controls. Namely, $\cA^N$ contains $N-$dimensional measurable functions $\a$ where $\a^i:[0,T]\times\dbR\mapsto A$ for $i=1,......,N$. In addition, we need $\a^i$ to be Lipschitz continuous on $x$ so that the SDE is well-posed. Similar conditions should hold for the set of admissible control ensemble $\cA$ and we postpone the definition of $\a\in\cA$ to the next subsection. 

Define functions $b:I\times[0,T]\times\dbR\times\cM_0\times A\mapsto\dbR$ and $\si:I\times[0,T]\times\dbR\times\cM_0\mapsto\dbR$. We need the following standing assumptions for the remainder of the paper to guarantee the well-posedness of the population dynamic. 
\begin{assum}\label{lipschitz}
 (i) For all $x_1,x_2\in\dbR$ $\mu_1,\mu_2\in\cM_0$, $a_1,a_2\in A$ and all $t\in[0,T]$, :
    $$|b(\th,t,x_1,\mu_1,a_1)-b(\th,t,x_2,\mu_2,a_2)|\leq L\big(|x_1-x_2|+d_1(\mu_1,\mu_2)+|a_1-a_2|\big)$$
    $$|\si(\th,t,x_1,\mu_1)-\si(\th,t,x_2,\mu_2)|\leq L\big(|x_1-x_2|+d_1(\mu_1,\mu_2)\big)$$
    for some Lipschitz constant $L$;\\
    (ii)There exists $\mu\in\cM_0$ such that :
    $$\sup_{\th,t,a}|b(\th,t,0,\mu,a)|+|\si(\th,t,0,\mu)|\leq L.$$
\end{assum}
\begin{rem}
    For simplicity, we consider 1-dimensional dynamic for each player in this paper. Our theory can be easily extended to d-dimensional dynamics. 
    
    Furthermore, the $N-$player games are formulated with $b,\si$ defined on the index space $I$ instead of $\{1,.....,N\}$ in order to compare with HMFG in Section \ref{sect-convergence}. One can easily state a equivalent $N-$player game with functions $b$ and $\si$ defined on the discrete index space.
\end{rem}

Let $W$ be a standard $N$-dimensional Brownian motion on filtered probability space $(\O,\cF,\dbP)$. Given initial conditions $(t,\bar x)\in[0,T]\times\dbR^N$, where $\bar x=(x_1,.......,x_n)$if the players use control $\a^N\in\cA^N$, the dynamic of the player $i\in\cC_l$ is the following
\begin{equation}\label{npd}
    X_{s,i}^{N,t,\bar x,\a^N}= x_i+\int_t^s b(\frac{l}{K},r,X_{r,i}^{N,t,\bar x,\a^N},\phi_r^{N},\a_{i,r}^{N})dr+\int_t^s \si(\frac{l}{K},r,X_{r,i}^{N,t,\bar x,\a^N},\phi_r^{N})dW_r^i
\end{equation}
where we denote $\a_{i,r}^{N}:=\a_{i}^{N}(r,X_{s,i}^{N,t,\bar x,\a^N})$ for simplicity and for each $\th\in I$,
$$\phi_s^{N}(\th)=\frac{1}{|\cC_{\left \lceil{\th K}\right \rceil }|}\sum_{j\in\cC_{\left \lceil{\th K}\right \rceil }}\d_{X_{s,j}^{N,t, \bar x,\a^N}}.$$
Notice that $\phi_s^{N}(\th)$ is a $\cP(\dbR)$ valued random variable and $\phi_s^{N}(\cdot)(\o)\in\cM_0$. The measurability is trivial since $\phi_s^{N}(\cdot)(\o)$ is a step function of $\th$.

Now let us define the utility function. The utility functions will be defined on $I$ instead of a discrete index space for the same reason as above. Let $F:I\times[0,T]\times\dbR\times\cM_0\times A\mapsto\dbR$ be the running cost function and $G:I\times\dbR\times\cM_0\mapsto\dbR$ be the terminal cost function. Throughout this paper, both $F$ and $G$ follow the standing assumptions below:
\begin{assum}\label{costassum}
    (i) F and G are uniformly Lipschitz continuous in $(x,\mu)$, i.e.: $\forall x_1,x_2\in\dbR$ $\mu_1,\mu_2\in\cM_0$,
    $$|F(\th,t,x_1,\mu_1,\a^\th)-F(\th,t,x_2,\mu_2,a)|\leq L(|x_1-x_2|+d_1(\mu_1,\mu_2))\qquad\forall(\th,t,a)\in I\times[0,T]\times A$$
    $$|G(\th,x_1,\mu_1)-G(\th,x_2,\mu_2)|\leq L(|x_1-x_2|+d_1(\mu_1,\mu_2))\qquad\forall\th\in I$$
    (ii)There exists $\mu\in\cM_0$:
    $$\sup_{\th,t,a}|F(\th,t,0,\mu,a)|+|G(\th,0,\mu)|\leq L$$
\end{assum} The payoff of player $i$ is:
\begin{equation}\label{npp}
    J^N(i,t,\bar x,\b^i,\a^{N,-i})=\dbE[G(\frac{l}{k},X_{s,i}^{N,t,\bar x,\b^i,\a^{N,-i}},\phi^{N}_T)+\int_t^T F(\frac{l}{k},s,X_{s,i}^{N,t,\bar x,\b^i,\a^{N,-i}},\phi_s^{N},\b_s^{i})ds]
\end{equation}
where $\b^i,\a^{N,-i}$ denote every control in vector $\a^N$ except for $i-th$ term is $\b^i$. Then the $N-$player game is defined in the usual way.
\begin{defn}{\textbf{($N-$Player Game)}}\label{nplayer}
    Given initial conditions $(t,\bar x)\in[0,T]\times\dbR^N$, a \textbf{$N-$player game system} is formulated with dynamic \eqref{npd} and payoff \eqref{npp}.
    A \textbf{Nash equilibrium (NE)} of the $N-$player game system is the control $\a^{*N}\in\cA^N$ such that for all $i=1,...,N$:
$$J^N(i,t,\bar x,\a^{*N})\geq J^N(i,t,\bar x,\b^i,\a^{*N,-i})\quad \forall\b^i\in\cA_i.$$
Denote $V^N(i,t,\bar x)=J^N(i,t,\bar x,\a^{*N})$ as the game value of player $i$.
\end{defn}
\begin{rem}
    Typically, the existence and the uniqueness of the equilibrium of the nonzero-sum $N-$player games are hard to analyze, especially when $N$ is large. To overcome the difficulty, people often consider the mean field model and use the mean field equilibrium as an approximation of the N-player game. It is often called the propagation of chaos and is shown in, e.g., \cite{cardaliaguet2019master, delarue2018}. However, this approach requires the homogeneous structure of the $N-$player games Nash equilibrium. The goal of this project is to propose a new heterogeneous mean field model that can describe the limit of the $N-$player game when $N$ goes to infinity. The relation between the Heterogeneous mean field model and the $N-$player games will be discussed in Section \ref{sect-convergence}.
\end{rem}
\begin{rem}
    Inspired by \cite{caines2019graphon}, we consider $N-$player games with $K-$types as the corresponding model of HMFG, unlike most GMFG models, see e.g. \cite{aurell2021stochastic,GMFG1,cao2025,Crucianelli24}, where each node of the graph represents a single player. Such structure works in GMFG setting because when taking $N\to\infty$, the model dependence on population states naturally becomes the dependence on population law ensemble thanks to the ELLN. However, our model relaxed from the aggregation term, the population law ensemble arises from the mean field limit of each type as number of players on each type and the number of types both go to infinity.
\end{rem}

\subsection{Heterogeneous Mean Field Games}
In this section, we formally define the Heterogeneous Mean Field Games (HMFG). Let $I=[0,1]$ denote the index set of uncountably many types of players. Each player depends only on the law ensemble of other types of players. Assume that within each type, the players use the same strategy, follow the same dynamic, and thus have the same law. Thus for each type of player, we only consider a single mean field dynamic that represents all the players of the same type. For simplicity, we call mean field dynamics of type $\th$ players as the player $\th$.

Consider the extended probability space $(\O\times I,\cF\boxtimes\cB_I,\dbP\boxtimes\l)$ introduced in Theorem \ref{fubiniEx} and the e.p.i. Brownian motion $B$ defined on it. In addition, $B$ is also independent with $W$ in the $N$ player game. Denote $\dbF^\th$ as the filtration generated by $B^\th$ and $\dbF$ as the filtration ensemble $\{\dbF^\th\}_{\th\in I}$. Next, we define the admissible strategy set.
\begin{defn}
    (i) Define the admissible closed-loop strategy for $\th$ as the function $\a^\th:[0,T]\times\dbR\to A$ such that $\a^\th(t,x)$ is uniformly Lipschitz continuous on $x$ with some Lipschitz constant $L_\a$, $\forall\th\in I,~t\in[0,T]$. Denote the set of admissible strategies for player $\th$ as $\cA^\th$;\\
    (ii) Define the admissible strategy ensemble as $\a=\{\a^\th\}_{\th\in I}$ where $\a^\th\in\cA^\th$ such that $\a^\cdot(t,x)$ is $\l_I-$measurable $\forall(t,x)\in[0,T]\times\dbR$. Denote the set of admissible strategy ensembles as $\cA$.
\end{defn}

Next we define the dynamics of the players. Function $b$ and $\si$ are the same as in last subsection. Given an initial time, an initial distribution and an admissible strategy profile $(t,\xi,\a)$, under Assumption \ref{lipschitz}, the following SDE ensemble has a unique solution in $L_\boxtimes^2(\Omega\boxtimes\mathcal{I};C([0,T],\dbR))$:
\begin{equation}\label{population}
    X_{s,\th}^{t,\xi,\a}=\xi^\th+\int_t^s b(\th,r,X_{r,\th}^{t,\xi,\a},\rho_r ^{t,\xi,\a},\alpha_r^\th)dr + \int_t^s \si(\th,r,X_{r,\th}^{t,\xi,\a},\rho_r ^{t,\xi,\a})dB_r^\th,\q\forall\th\in I,
\end{equation}
where $\rho_s ^{t,\xi,\a}:=\cL(X_s^{t,\xi,\a})$, and $X_{s,\th}^{t,\xi,\a}$ is the dynamic of player $\th\in I$. We can take $\xi^\th\equiv x^\th\in \dbR$ indicating deterministic initial positions. Given the population law ensemble, individual player can employ his own strategy with different initial states $\z^\th$ and strategy $\b^\th\in\cA^\th$:
\begin{equation}
    X_{s,\th}^{t,\z^\th,\xi,\b^\th,\a}=\z^\th+\int_t^s b(\th,r,X_{r,\th}^{t,\z^\th,\xi,\b^\th,\a},\rho_r ^{t,\xi,\a},\b_r^\th)dr + \int_t^s \si(\th,r,X_{r,\th}^{t,\z^\th,\xi,\b^\th,\a},\rho_r ^{t,\xi,\a})dB_r^\th
\end{equation}

Next, we define the utility function. Let the same $F$ and $G$ from the previous subsection be the running cost function and the terminal cost function in HMFG. 
\begin{defn}
    Denote $J(\th,t,x,\b^\th;\a,\xi)$ as the payoff of player $\th$ starting at time $t$ and position $x$ when the strategy of player $\th$ is $\b^\th\in\cA^\th$ and the initial state ensemble of population is $\xi\in\dbL_{\boxtimes}^2(\O\times I,\cF_t)$ and the strategy ensemble of population is $\a\in\cA$:
\begin{equation}\label{control}
    J(\th,t,x,\b^\th; \a, \xi)=\dbE\Big[G(\th,X_{T,\th}^{t,x,\xi,\b^\th,\a},\rho_T ^{t,\xi,\a})+\int_t^T F(\th,s,X_{s,\th}^{t,x,\xi,\b^\th,\a},\rho_s ^{t,\xi,\a},\b_s^\th)ds\Big]
\end{equation}
\end{defn}

Since $J(\th,t,x,\b^\th; \a, \xi)$ only depends on $\xi$ through the law of $\cL(X_s^{t,\xi,\a})$, we can take any initial population state $\xi'\in\dbL_\boxtimes^2(\O\times I;\cF_t)$ such that $\cL(\xi')=\cL(\xi)$ and have $\cL(X_s^{t,\xi,\a})=\cL(X_s^{t,\xi',\a})$. This implies that $J(\th,t,x,\b^\th, \a, \xi)=J(\th,t,x,\b^\th, \a, \xi')$. Thus it suffices to define 
$$J(\th,t,x,\b^\th; \a, \mu):=J(\th,t,x,\b^\th; \a, \xi)\quad\forall\xi\in\dbL_\boxtimes^2(\O\times \cI;\cF_t)\text{  such that  }\cL(\xi)=\mu$$
Since $J$ depends on $\a, \mu$ through $\cL(X^{t,\xi,\a})$, we can also denote $$J(\th,t,x,\b^\th;\cL(X_{[t,T]}^{t,\xi,\a}))=J(\th,t,x,\b^\th; \a, \rho).$$
$J(\th,t,x,\b^\th;\cL(X_{[t,T]}^{t,\xi,\a}))$ also indicate that when the population law ensemble is fixed, the player $\th$ solves an optimal control problem. For a fixed population strategy ensemble $\a\in\cA$ and initial law ensemble $\mu\in\cM_0$, the optimal payoff of player $\th$ is
$$V(\th,t,x;\a,\mu)=\sup_{\b^\th\in\cA^\th}J(\th,t,x,\b^\th; \a, \mu)$$
Then we define the Nash Equilibrium of a heterogeneous mean field game as the following.
\begin{defn}\label{GMFGNE}
    \textbf{(Heterogeneous Mean Field Game)} Given initial conditions $(t,x)$ of player $\th$ and initial law ensemble of the population $\rho$, a Heterogeneous mean field game is formulated with dynamic \eqref{population} and payoff \eqref{control}.
    
    A control $\a^{*}\in\cA$ is a Nash Equilibrium (NE) of the HMFG if:
$$J(\th,t,x,\a^{\th*};\a^*, \rho)\geq J(\th,t,x,\b^\th;\a^*, \rho)\qquad \forall \b^\th\in\cA^\th\qquad\forall\th\in I,$$
and the corresponding equilibrium payoff of the HMFG is:
$$V(\th,t,x;\a^*,\mu):=\sup_{\b^\th\in\cA^\th}J(\th,t,x,\b^\th; \a^*, \rho)=J(\th,t,x,\a^{\th*},\a^*, \rho).$$ 
We denote the game value of the HMFG as $V(\th,t,x,\mu):=V(\th,t,x;\a^*,\mu)$
\end{defn}
\subsection{HMFG System: PDE Formulation}
In Definition \ref{GMFGNE}, the equilibrium of HMFG is defined as a fixed point of admissible strategy ensemble. We can also define the equilibrium of HMFG as a fixed point on population law ensemble.
\begin{defn}\label{fixedlaw}
Given initial conditions $(t,x)$ and the law ensemble of the population $\mu$, we say $\rho\in\cM$ is an equilibrium of HMFG if:

1) For a fixed population law ensemble $\rho\in\cM$ , where $\rho_t=\mu$, each player solves the optimize the payoff $J(\th,t,x,\b^\th; \rho)$ and finds optimal control $\a^{\th*}$;

2) For a given admissible strategy $\a^{*}\in\cA$, $\cL(X_{\th}^{t,\mu,\a^{*}})=\rho_\th$.
\end{defn}
\begin{rem}
    Definition \ref{GMFGNE} and \ref{fixedlaw} are equivalent. The benefit of Definition \ref{fixedlaw} which defines the equilibrium as a fixed point on population law ensemble is that when each player solves an optimal control problem given population law ensemble, there can be multiple optimal controls. However, as we will show later, the population law ensemble when applying the optimal strategies is unique. The uniqueness of equilibrium that we show in the paper will be the uniqueness of fixed point $\rho\in\cM$.
\end{rem}

Similar to MFG, we can also derive an infinite-dimensional PDE system. First, let us define 
$$h(\th,t,x,p,\rho,a^\th):=b(\th,t,x,\rho,a^\th)p+F(\th,t,x,\rho,a^\th)$$
and Hamiltonian:
$$H(\th,t,x,p,\rho):=\sup_{a^\th\in A}h(\th,t,x,p,\rho,a^\th)$$
Fix population law ensemble $\bar \rho_{[t,T]}\in\cM$. For each player $\th$ starting at $(t,x)$, we need to solve an optimal control problem. Applying standard approach in the optimal control theory, we have an HJB equation for each $\th\in I$:
\begin{equation}\label{HFixedlaw}
    \begin{cases}
        \partial_t u_\th^{\bar\rho} + \frac{1}{2}|\si(\th,t,x,\bar\rho)|^2\partial_{xx} u_\th^{\bar\rho}+H(\th,t,x,\partial_x u_\th^{\bar\rho},\bar \rho_t)=0\\
        u_\th^{\bar\rho}(T,x) = G(\th, x, \bar \rho_T)\qquad u_\th^{\bar\rho}(s,x)\in[t,T]\times\dbR
    \end{cases}    
\end{equation}
where $u_\th^{\bar\rho}(t,x)=\sup_{\b^\th} J(\th,t,x;\bar\rho_{[t,T]})$. For $u_\th^{\bar\rho}(t,x)$ to have a classical solution we need to following assumptions.
\begin{assum}\label{c2}
    (i) $\si(\th,s,\cdot,\mu)\in C_b^2$: for all $(\th,s,\mu)$, $\si$ is continuously differentiable over $x$ and $\si,\pa_x \si$ are uniformly Lipschitz with coefficient $L$ and $|\si(t,0)|,|\pa_x \si(t,0)|$ are uniformly bounded with constant $L$.\\
    (ii) $|\si(\th,s,x,\mu)|\geq \frac{1}{L}$ for all $(\th,s,x,\mu)$.\\
    (iii) $H$ and $\pa_pH$ is uniformly Lipschitz in $(x,p)$ with Lipschitz constant $L$.
\end{assum}
Then by Theorem 2.4.1 in \cite{ZhangThesis}, $u_\th^{\bar\rho}$ has a classical solution in $C_b^2$ and $|\pa u_\th^{\bar\rho}|\leq C$ for some $C$ only depends on $T$ and $L,$. 

Suppose $\exists \a^*\in\cA$ such that 
\begin{equation}\label{hNE}
    h(\th,t,x,\pa_x u_\th^{\bar\rho},\bar\rho_t,a^{*\th}(t,x))=H(\th,t,x,\pa_x u_\th^{\bar\rho},\bar\rho_t)\q\forall(t,x)\in[0,T]\times\dbR\q\forall\th\in I
\end{equation}
Since $b$, $F$, $\pa u_\th^{\bar\rho}$ are continuous, by classical envelope theorem, see e.g. Corollary 299 in \cite{Border}, we have 
\begin{equation}\label{bpH}
    b(\th,t,x,\bar\rho,a^{*\th}(t,x))=\pa_pH(\th,t,x,\pa_x u_\th^{\bar\rho},\bar\rho_t).
\end{equation}
and the dynamic of $X_{t,\th}^{t,\xi,\a^{*}}$:
\begin{equation}\label{dynamicfixlaw}
\begin{aligned}
    X_{s,\th}^{t,\xi,\a^{\th*}}=&\xi+\int_t^s \pa_p H(\th,r,X_{r,\th}^{t,\xi,\a^{*\th}},\pa_xu_\th^{\bar\rho}(r,X_{r,\th}^{t,\xi,\a^{*\th}}),\bar\rho_r)dr + \int_t^s \si(\th,r,X_{r,\th}^{t,\xi,\a^{*\th}},\bar\rho_r)dB_r^\th
\end{aligned}
\end{equation}
Denote the law $\rho_t^\th=\mathcal{L}(X_{t,\th}^{t,\xi,\a^{\th*}})$. We can also derive the forward PDE through Fokker Plank equations. Since $\rho$ is a equilibrium, we have $\rho\equiv\bar\rho$. Combining with the HJB equation \eqref{HFixedlaw} above, we get a coupled forward backward PDE (FBPDE)
\begin{equation}\label{fbpde}
    \begin{cases}
        \partial_t u_\th+\frac{1}{2}|\si(\th,s,x,\rho_s)|^2\partial_{xx} u_\th+H(\th,s,x,\partial_x u_\th, \rho_s)=0\\
        \partial_t\rho_\th +\partial_x(\partial_p H(\th,s,x,\partial_x u_\th,\rho_s)\rho_{s,\th})-\frac{1}{2}\partial_{xx}(|\si(\th,s,x,\rho_s)|^2\rho_{t,\th})=0\\
        u_\th(T,x) = G(\th, x, \rho_T),\rho_{s,\th}=\mu^\th,\quad \forall s\in[t,T],\quad\forall\th\in I.
    \end{cases}    
\end{equation}
Notice that for each player $\th$, the solution $(u_\th,\rho^\th)$ pair depends on all the other players through $\rho$, thus the FBPDE system above is an coupled infinite-dimensional system. 
\begin{defn}\label{pdesol}
    Define a solution of FBPDE \eqref{fbpde} to be the pair $(u,\rho)$ such that for all $\th\in I$, $u^\th$ is a classical solution of the HJB equation in \eqref{fbpde} and $\rho$ is a distributional solution of the Fokker-Planck equation in \eqref{fbpde}.
\end{defn}
Furthermore, we have the following characterization theorem.
\begin{thm}\label{pdewellposed}
    Assume Assumption \ref{c2} holds. Given initial time $t$ and population initial law ensemble $\mu$, the game has an equilibrium $\rho\in\cM$ if and only if the FBPDE system \eqref{fbpde} has a solution and there exists an optimizer $\a^{*}(t,x)\in\cA$ such that \eqref{hNE} hold. Furthermore $V(\th,t,x,\mu)=u_\th(t,x)$.
\end{thm}
\begin{proof}
    First we assume given initial time $t$ and population initial law ensemble $\rho_0$, the game has a equilibrium $\rho\in\cM$. By the derivation above, it suffices to show that given $\rho$, the optimal closed-loop strategy $\a^{*\th}$ for each player $\th$ satisfies \eqref{hNE}. By dynamic programming principle in optimal control theory, see e.g. Theorem 3.3.1 in \cite{Pham2009}, we have
\begin{equation*}
    u^\rho_\th(t,x)=\dbE\Big[u^\rho_\th(t+h,X_{t+h,\th}^{t,x,\xi,\a^*})+\int_t^{t+h} F(\th,r,X_{r,\th}^{t,x,\xi,\a^*},\rho_r ,\a_r^{*\th}(r,X_{r,\th}^{t,x,\xi,\a^*}))dr\Big]
\end{equation*}
Take $h\to0$ we have:
\begin{equation}\label{alphaOpt}
    \begin{aligned}
        0=&\partial_t u_\th^{\rho} + \frac{1}{2}|\si(\th,t,x,\rho)|^2\partial_{xx} u_\th^{\rho}+b(\th,t,x,\a^{*\th}(t,x) \rho_t))\partial_x u_\th^{\rho}-F(\th,t,x,\rho_t,\a^{*\th}(t,x))\\
        =&\partial_t u_\th^{\rho} + \frac{1}{2}|\si(\th,t,x,\rho)|^2\partial_{xx} u_\th^{\rho}+H(\th,t,x,\partial_x u_\th^{\rho}, \rho_t)
    \end{aligned}
\end{equation}
and we are done.

    For the other direction. Assume FBPDE system \eqref{fbpde} on time interval $[t,T]$ with $\mu$ as the initial condition for $\rho_0$ has a solution $(u,\rho)$ as defined in \ref{pdesol} and a closed-loop control $\a^*$ exist such that \eqref{hNE} holds. Then by standard verification theorem in control theory, $u_\th(t,x)$ is the optimal value of the player $\th$ when the population law is $\rho$ and $\a^*$ is the corresponding control. Since \eqref{hNE} is true, \eqref{bpH} also hold true. Then the dynamic of player $\th$ using strategy $\a^*$ follows \eqref{dynamicfixlaw}, and $\cL(X_{t,\th}^{t,\xi,\a^{*}})\equiv\rho_{\th,t}$. Thus $\rho$ satisfies the definition of an equilibrium, and $u_\th(t,x)=V(\th,t,x,\mu)$.
\end{proof}

\subsection{HMFG System: FBSDE Formulation}\label{fbsde fp}
 Now we derive the coupled infinite-dimensional FBSDE system for the HMFG at the equilibrium. We follow the same fixed point argument on the population law ensemble as in Definition \ref{fixedlaw}. 

Let $\rho\in \cM$ be an equilibrium and $\a^{*\th}\in\cA^\th$ be the corresponding optimal control for player $\th$ when population law ensemble is fixed. By the argument in the proof of Theorem \ref{pdewellposed}, $\a^*$ satisfies \eqref{hNE}, dynamic $X_{s,\th}^{t,\xi,\a^*}$ solves SDE \eqref{dynamicfixlaw}, and $\cL(X_{s,\th}^{t,\xi,\a^*})\equiv\rho_{\th,s}$ for all $(\th,s)$. By Theorem \ref{pdewellposed}, infinite-dimension FBPDE system \eqref{fbpde} has a classical solution $(u,\rho)$. Let $Y_{s,\th}^{t,\xi,\a^*}:=u_\th(s,X_{s,\th}^{t,\xi,\a^*})$ and $Z_{s,\th}^{t,\xi,\a^*}:=\pa_xu_\th(s,X_{s,\th}^{t,\xi,\a^*})\si(s,X_{s,\th}^{t,\xi,\a^*},\rho_s)$. Then $(Y_{\th}^{t,\xi,\a^*},Z_{\th}^{t,\xi,\a^*})$ solves the following BSDE:
\begin{equation}\label{cost}
    \begin{aligned}
        Y_{s,\th}^{t,\xi,\a^\th}=&G(\th,X_{T,\th}^{t,x,\a^\th},\rho_T)+\int_s^TH(\th,r,X_{r,\th}^{t,x,\a^\th},\pa_xu_\th(r,X_{r,\th}^{t,\xi,\a^*}), \rho_r)\\
    &-\pa_pH(\th,r,X_{r,\th}^{t,x,\a^\th},\pa_xu_\th(r,X_{r,\th}^{t,\xi,\a^*}), \rho_r)\pa_xu_\th(r,X_{r,\th}^{t,\xi,\a^*})dr-\int_s^TZ_{r,\th}^{t,x,\a^\th}dB_r^\th\\
    =&G(\th,X_{T,\th}^{t,x,\a^\th},\rho_T)+\int_s^TF(\th,r,X_{r,\th}^{t,\xi,\a^{*\th}}, \rho_r,\a^{*\th}(r,X_{r,\th}^{t,\xi,\a^{*\th}}))dr-\int_s^TZ_{r,\th}^{t,\xi,\a^{\th*}}dB_r^\th
    \end{aligned}
\end{equation}
The last equality holds because $\a^*$ satisfies \eqref{hNE} which implies:
\begin{equation}\label{F}
    F(\th,t,x,a^{*\th}(t,x))=H(\th,t,x,\pa_xu_\th(t,x),\rho_t)-\partial_p H (\th,t,x,\pa_xu_\th(t,x),\rho_t)\pa_xu_\th(t,x)
\end{equation}
and we have the optimal value $u_\th(t,x)=Y_{t,\th}^{t,\xi,\a^{*\th}}$. 

Plugin $Z_{s,\th}^{t,\xi,\a^*}:=\pa_xu_\th(s,X_{s,\th}^{t,\xi,\a^*})\si(s,X_{s,\th}^{t,\xi,\a^*},\rho_s)$ in \eqref{dynamicfixlaw} and \eqref{cost}, we get a coupled FBSDE for player $\th$
\begin{equation}\label{coupledfbsde}
    \begin{cases}
        X_{s,\th}^{t,\xi,\a^*}=&\xi^\th+\int_t^s \hat b(\th,r,X_{r,\th}^{t,\xi,\a^*},Z_{r,\th}^{t,\xi,\a^*},\rho_r)dr + \int_t^s \si(\th,r,X_{r,\th}^{t,\xi,\a^*},\rho_r)dB_r^\th\\
        Y_{s,\th}^{t,\xi,\a^*}=&G(\th,X_{T,\th}^{t,\xi,\mu},\rho_T)+\int_s^T \hat F(\th,r,X_{r,\th}^{t,\xi,\a^*},Z_{r,\th}^{t,\xi,\a^*},\rho_r)dr-\int_s^T Z_{r,\th}^{t,\xi,\a^*} dB_r^\th
    \end{cases}
\end{equation}
Recall that $\cL(X_{s,\th}^{t,\xi,\a^*})=\rho_s$. We have the following infinite-dimensional coupled FBSDE
\begin{equation}\label{fbsdesystem}
    \begin{cases}
        X_{s,\th}^{t,\xi,\mu}=&\xi^\th+\int_t^s \hat b(\th,r,X_{r,\th}^{t,\xi,\mu},Z_{r,\th}^{t,\xi,\mu},\cL(X_{r}^{t,x,\mu}))dr + \int_t^s \si(\th,r,X_{r,\th}^{t,\xi,\mu},\cL(X_{r}^{t,\xi,\mu}))dB_r^\th\\
        Y_{s,\th}^{t,\xi,\mu}=&G(\th,X_{T,\th}^{t,\xi,\mu},\cL(X_{T}^{t,\xi,\mu}))+\int_s^T \hat F(\th,r,X_{r,\th}^{t,\xi,\mu},Z_{r,\th}^{t,\xi,\mu},\cL(X_{r}^{t,\xi,\mu}))dr-\int_s^T Z_{r,\th}^{t,\xi,\mu} dB_r^\th
    \end{cases}
\end{equation}
where 
$$\hat F(\th,t,x,z,\rho)=H(\th,t,x,\si^{-1}(\th,x,\rho)z,\rho)-\partial_p H (\th,t,x,\si^{-1}(\th,x,\rho)z,\rho)\si^{-1}(\th,x,\rho)z$$
$$\hat b(\th,t,x,z,\rho)=\pa_pH(\th,t,x,\si^{-1}(\th,x,\rho)z,\rho)$$
Here $X_{\th}^{t,\xi,\a^{*}}$ is the dynamic and $Y_{t,\th}^{t,\xi,\a^{*}}$ is the payoff of the player $\th$ using NE strategy. 

Notice that Assumption \ref{c2} is not enough to claim that $\hat F$ and $\hat b$ are uniformly Lipschitz in $(x,\rho)$. However, it is true when $z$ has a compact support. We will show later that $Z_{r,\th}^{t,x,\a^{\th*}}$ indeed has a compact support when we show the well-posedness of \eqref{fbsdesystem}, thus it is enough to consider $\hat F$ and $\hat b$ on a compact support. Here, we will show a characterization theorem.
\begin{thm}\label{fbsdeNE}
     Assume Assumption \ref{c2} holds. Given an initial time $t$ and population initial law ensemble $\mu$, the game has a equilibrium $\rho$ if and only if coupled infinite-dimensional FBSDE system \eqref{fbsdesystem} has a unique solution and there exists an optimizer $\a^{*\th}(t,x)$ such that \eqref{hNE} hold. Furthermore $V(\th,t,x,\mu)=Y_{t,\th}^{t,x,\mu}$.
\end{thm}
\begin{proof}
    The forward direction is showed in the derivation above. For the backward direction, suppose infinite-dimensional FBSDE \eqref{fbsdesystem} has a solution $(X_{\th}^{t,\xi,\mu},Y_{\th}^{t,\xi,\mu},Z_{\th}^{t,\xi,\mu})$. Denote $\rho_{\th,s}=\cL(X_{s,\th}^{t,\xi,\mu})$. Under Assumption \ref{c2}, the HJB equation \eqref{HFixedlaw} has a unique solution. Since $\a^{*\th}(t,x)$ satisfies \eqref{hNE}, by standard verification theorem, see e.g. Theorem 3.5.2 in \cite{Pham2009}, $\a^*(s,X_{s,\th}^{t,\xi,\a^*})$ is an optimal control for $\th$ with fixed population law $\rho$. Furthermore, $X_{s,\th}^{t,\xi,\a^*}$ solves the SDE \eqref{dynamicfixlaw}.  Let $Y_{s,\th}^{t,\xi,\a^*}:=u_\th(s,X_{s,\th}^{t,\xi,\a^*})$ and $Z_{s,\th}^{t,\xi,\a^*}:=\pa_xu_\th(s,X_{s,\th}^{t,\xi,\a^*})\si(s,X_{s,\th}^{t,\xi,\a^*},\rho_s)$. Then $(X_{\th}^{t,\xi,\a^*},Y_{\th}^{t,\xi,\a^*},Z_{\th}^{t,\xi,\a^*})$ solves the FBSDE 
    \begin{equation}\label{fixedthfbsde}
    \begin{cases}
        X_{s,\th}^{t,\xi,\a^*}=&\xi^\th+\int_t^s \hat b(\th,r,X_{r,\th}^{t,\xi,\a^*},Z_{r,\th}^{t,\xi,\a^*},\rho_{\th,r}))dr + \int_t^s \si(\th,r,X_{r,\th}^{t,\xi,\a^*},\rho_{\th,r})dB_r^\th\\
        Y_{s,\th}^{t,\xi,\a^*}=&G(\th,X_{T,\th}^{t,\xi,\a^*},\rho_{\th,T})+\int_s^T \hat F(\th,r,X_{r,\th}^{t,\xi,\a^*},Z_{r,\th}^{t,\xi,\a^*},\rho_{\th,r})dr-\int_s^T Z_{r,\th}^{t,\xi,\a^*} dB_r^\th
    \end{cases}
\end{equation}
Furthermore, by Theorem 8.3.1 in \cite{Zhangbook}, since $u_\th^{\rho}$ has a classical solution, FBSDE \eqref{fixedthfbsde} has a unique solution. Then $(X_{\th}^{t,\xi,\a^*},Y_{\th}^{t,\xi,\a^*},Z_{\th}^{t,\xi,\a^*})=(X_{\th}^{t,\xi,\mu},Y_{\th}^{t,\xi,\mu},Z_{\th}^{t,\xi,\mu})$. In particular, $\cL(X_{\th}^{t,\xi,\a^*})=\cL(X_{\th}^{t,\xi,\mu})$. Thus $\cL(X_{\th}^{t,\xi,\mu})$ is indeed an equilibrium by Definition \ref{fixedlaw}.\end{proof}

\section{Local Well-posedness}\label{sect-smallT}
We will show the well-posedness of the infinite-dimensional FBSDE system (\ref{fbsdesystem}) on small time interval. We will follow the fixed point argument outlined above and show that a unique fixed point exists via the Banach fixed point theorem. For the Banach fixed point argument, we first fix $\bar\rho_1,\bar\rho_2 \in\cM$, and then find the solution of the HJB equation \eqref{HFixedlaw} $u^{\bar\rho_1}$, and $u^{\bar\rho_2}$. The fixed point exists if $d_2(\rho^{\bar\rho_1},\rho^{\bar\rho_2})\leq Cd_2(\bar\rho_1,\bar\rho_2)$ for some $C<1$. 

\begin{thm}\label{wellposed}
    Under Assumption \ref{c2}, FBSDE \eqref{fbsdesystem} has a unique solution when $T\leq\d_0$ where $\d_0$ only depends on $L$.
\end{thm}
\begin{proof} We will show by the Banach fixed point theorem. For simplicity, we will show for initial conditions $(0,x,\mu)$. Fix a law ensemble $\bar \rho\in \cM$ such that $\bar\rho_0=\mu$. We first show that FBSDE \eqref{coupledfbsde} has a unique solution $(X^{\bar\rho},Y^{\bar\rho},Z^{\bar\rho})$ on $\dbL_\boxtimes^2\times\dbL_\boxtimes^2\times\dbZ$ where $\dbZ$ is a subspace of $\dbL_\boxtimes^2$ that we will define later. Denote mapping $\Phi:\cM\mapsto\cM$ as $\Phi(\bar\rho)(\th)=\cL(X^{\bar\rho,\th})$. Since $X^{\bar\rho,\cdot}$ is $\cI-$measurable, $\Phi(\bar\rho)(\cdot)$ is also $\cI-$measurable. We will show that $\Phi$ is a contraction map on $\cM$ when $T$ is small enough.

First, define the subset 
$$\dbZ:=\{Z\in\dbL_\boxtimes^2(\O\times I;C([0,T];\dbR)): |Z_t^\th|\leq M ~~\forall t\in[0,T]~~\dbP\boxtimes\l_I-a.s.\}.$$
where $M\geq|\si(\th,s,x,\rho_s)\pa_xu_\th^\rho(s,x)|$ for all $(\th,s,x,\rho)$, here $\rho\in\cM$. By Assumption \ref{c2} such $M$ exists and depends on $T$ and $L$. $\dbZ$ is a complete subspace: for a Cauchy sequence $\{Z^n\}\subset\dbZ$, there exists a $Z\in\dbL_\boxtimes^2(\O\times I;C([0,T];\dbR))$ such that $\lim_{n\to\infty}Z_n=Z$ in $\|\cdot\|_2$ by completeness of $\dbL_\boxtimes^2(\O\times I;C([0,T];\dbR))$. Then there exists a subsequence $\{Z^{n_k}\}$ such that $|Z_t^\th|=\lim_{k\to\infty}|Z_t^{n_k,\th}|\leq M$ $[0,T]\times\O\times I$ a.e. 

Now we will show that for fixed $\bar\rho\in\cM$ such that $\rho_0=\mu$, the FBSDE \eqref{coupledfbsde} has a unique solution $(X^{\bar\rho},Y^{\bar\rho},Z^{\bar\rho})$ on $\dbL_\boxtimes^2\times\dbL_\boxtimes^2\times\dbZ$. However, the difficulty lies in the fact that $\hat b$ and $\hat F$ are not Lipschitz in $x$. To overcome that, we need to first consider an truncated version of equation \eqref{coupledfbsde}. Then we observe that it's decoupling field coincide with HJB equation \eqref{HFixedlaw} and use the fact that when the solution of HJB equation \eqref{HFixedlaw} exist, it's derivative on $x$ is bounded.

Define $\psi^M(x)=-M\vee x\wedge M$
$$\hat b^M(\th,t,x,z,\mu)=\hat b(\th,t,x,\psi^M(z),\mu)\quad\hat F^M(\th,t,x,z,\mu)=\hat F(\th,t,x,\psi^M(z),\mu)$$
under the truncation on $z$, $\hat b^M$ and $\hat F^M$ are {Lipschitz} on $(x,p)$ since $H$, and $\pa_pH$ are bounded and Lipschitz in $(x,p)$ and $\si^{-1}$ is bounded and Lipschitz in $x$ and $z$ has bounded support $[-M,M]$. Let us consider the truncated version of FBSDE system \eqref{fbsdesystem}:
\begin{equation}\label{truncatedcoupledfbsde}
    \begin{cases}
        X_{s,\th}^{x,\bar\rho,M}&=x+\int_0^s \hat b^M(\th,r,X_{r,\th}^{x,\bar\rho,M},Z_{r,\th}^{x,\bar\rho,M},\bar\rho_r)dr + \int_0^s \si(\th,r,X_{r,\th}^{x,\bar\rho,M},\bar\rho_r)dB_r^\th\\
        Y_{s,\th}^{x,\bar\rho,M}&=G(\th,X_{T,\th}^{x,\bar\rho,M},\bar\rho_T)+\int_s^T \hat F^M(\th,r,X_{r,\th}^{x,\bar\rho,M},Z_{r,\th}^{x,\bar\rho,M},\bar\rho_r)dr-\int_s^T Z_{r,\th}^{x,\bar\rho,M} dB_r^\th
    \end{cases}
\end{equation}
By the standard contraction mapping argument $(X^{x,\bar\rho,M},Y^{x,\bar\rho,M},Z^{x,\bar\rho,M})$ has a unique strong solution on $\dbL_\boxtimes^2\times\dbL_\boxtimes^2\times\dbL_\boxtimes^2$ when $T$ is small enough.

For each $\th\in I$, let us consider the decoupling field $u_\th^{\bar\rho,M}(t,x)$ such that $Y_{s,\th}^{x,\bar\rho,M}=u_\th^{\bar\rho,M}(s,X_{s,\th}^{x,\bar\rho,M})$. Apply It\^o's lemma, $u_\th^{\bar\rho,M}(t,x)$ satisfies the following HJB equation:
\begin{equation}\label{truncatedHJB}
    \begin{cases}
        \partial_t u_\th^{\bar\rho,M}+\frac{1}{2}|\si_\th^{\bar\rho}|^2\partial_{xx}u_\th^{\bar\rho,M}+H(\th,t,x,(\si_\th^{\bar\rho})^{-1}\psi^M(\si_\th^{\bar\rho}\pa_xu_\th^{\bar\rho,M}),\bar\rho_t)\\
        ~~~+\pa_pH(t,x,(\si_\th^{\bar\rho})^{-1}\psi^M(\si_\th^{\bar\rho}\pa_xu_\th^{\bar\rho,M}),\bar\rho_t)(\pa_xu_\th^{\bar\rho,M}-(\si_\th^{\bar\rho})^{-1}\psi^M(\si_\th^{\bar\rho}\pa_xu_\th^{\bar\rho,M}))=0\\
        u_\th^{\bar\rho,M}(T,x) = G(\th, x, \bar \rho_T),\quad u_\th^{\bar\rho,M}(s,x)\in[0,T]\times\dbR.
    \end{cases}    
\end{equation}
Since $H(\th,t,x,-\frac{M}{\si_\th^{\bar\rho}}\vee p\wedge \frac{M}{\si_\th^{\bar\rho}},\bar\rho_t)+\pa_pH(t,x,-\frac{M}{\si_\th^{\bar\rho}}\vee p\wedge \frac{M}{\si_\th^{\bar\rho}},\bar\rho_t)(p-(-\frac{M}{\si_\th^{\bar\rho}}\vee p\wedge \frac{M}{\si_\th^{\bar\rho}})$ is uniformly Lipschitz in $(x,p)$ in addition to Assumption \ref{c2}, quasi-linear pde \eqref{truncatedHJB} has a unique classical solution. 

Furthermore, we know the solution of HJB equation \ref{HFixedlaw} $u_\th^{\bar\rho}(t,x)$ has a uniformly bounded partial derivative such that 
$|\si_\th^{\bar\rho}\pa_xu_\th^{\bar\rho}|\leq M$ uniformly. Thus $u_\th^{\bar\rho}(t,x)$ is a solution of \eqref{truncatedHJB} and 
\begin{equation}\label{zNu}
    |Z_{t,\th}^{x,\bar\rho,M}|=|\si(\th,t,X_{t,\th}^{x,\bar\rho,M},\bar\rho_t)\pa_xu_\th^{\bar\rho}(t,X_{t,\th}^{x,\bar\rho,M})|\leq M\quad\forall t\in[0,T]\quad\dbP-a.s.\quad \forall\th\in I.
\end{equation}
Since $Z^{x,\bar\rho,M}$ never reaches the truncation $M$, we can conclude that $(X^{x,\bar\rho,M},Y^{x,\bar\rho,M},Z^{x,\bar\rho,M})$ is in fact a solution of FBSDE \eqref{coupledfbsde} and $(X^{x,\bar\rho,M},Y^{x,\bar\rho,M},Z^{x,\bar\rho,M})\in\dbL_\boxtimes^2\times\dbL_\boxtimes^2\times\dbZ$.

Next, we will show that FBSDE \eqref{coupledfbsde} has a unique solution on $\dbL_\boxtimes^2\times\dbL_\boxtimes^2\times\dbZ$. Suppose $(X^1,Y^1,Z^1)$, $(X^2,Y^2,Z^2)\in\dbL_\boxtimes^2\times\dbL_\boxtimes^2\times\dbZ$ are two solutions for FBSDE \eqref{coupledfbsde}. Denote $\D X=X^1-X^2$ and similar for $\D Y$ and $\D Z$. Then for $\l_I$-almost all $\th\in I$, we have FBSDE:
\begin{equation}
    \begin{cases}
        \D X_{t}^\th=\int_0^t\b_{s,\th}^1\D X_s^\th+\g_{s,\th}^1\D Z_s^\th ds+\int_0^t\b_{s,\th}^2\D X_s^\th dB_s^\th\\
        \D Y_t^\th = \b_{T,\th}^3\D X_T^\th+\int_t^T\b_{s,\th}^4\D X_s^\th+\g_{s,\th}^4\D Z_s^\th ds-\int_t^T \D Z_s^\th dB_s^\th
    \end{cases}
\end{equation}
where 
$$\b_{s,\th}^1=\frac{\hat b(\th,s,X_s^{1,\th},Z_s^{1,\th},\rho_s)-\hat b(\th,s, X_s^{2,\th},Z_s^{1,\th},\rho_s)}{\D X_s^\th}\dbI_{\{|\D X_s^\th|>0\}}$$
$$\g_{s,\th}^1=\frac{\hat b(\th,s, X_s^{2,\th},Z_s^{1,\th},\rho_s)-\hat b(\th,s,X_s^{1,\th},\bar Z_s^{2,\th},\rho_s)}{\D Z_s^\th}\dbI_{\{|\D Z_s^\th|>0\}}$$
and $\b_{s,\th}^2$ for $\si$, $\b_{T,\th}^3$ for $G$, and $(\b_{s,\th}^4,\g_{s,\th}^4)$ for $\hat F$ respectively. Since $\hat F$ and $\hat b$ are Lipschitz in $(x,z)$ when $z$ has compact support, $|\b_{s,\th}^1|,|\b_{s,\th}^2|,|\b_{s,\th}^3|,|\b_{s,\th}^4|,|\g_{s,\th}^1|,|\g_{s,\th}^4|\leq L$ where $L$ is the common Lipschitz constant of $G,\hat F$ and $\hat b$ over $(x,z)$. By the a priori bound of FBSDE, we have:
$$\dbE[\sup_{t\in[0,T]}|\D X_t^\th|^2+\sup_{t\in[0,T]}|\D Y_t^\th|^2+\int_0^T|\D Z_s^\th|^2ds]d\leq0$$
Thus $(X^{1,\th},Y^{1,\th},Z^{1,\th})\equiv(X^{2,\th},Y^{2,\th},Z^{2,\th})$ $\forall\th\in I$ $\l_I$-almost surely. 

Lastly, we will show that the map from $\bar\rho$ to $\cL(X_{\cdot}^{\bar\rho})$ is a contraction map.
Fix initial conditions $(0,\mu)$, for the population law $\rho,\bar\rho\in\cM$, where $\rho_0=\bar\rho_0=\mu$, denote the solution of corresponding FBSDE \eqref{coupledfbsde} with initial conditions $(0,\xi)$ where $\cL(\xi)=\mu$ as $(X,Y,Z)$ and $(\bar X,\bar Y,\bar Z)$. We need to find $d_2(\cL(\bar X),\cL(X))\leq Cd_x(\bar\rho,\rho)$ for some $C<1$. Let us denote $(\D X,\D Y,\D Z)=(X-\bar X,Y-\bar Y,Z-\bar Z)$, then we have
\begin{equation}
    \begin{cases}
        \D X_{t}^\th=\int_0^t\b_{s,\th}^1\D X_s^\th+\g_{s,\th}^1\D Z_s^\th+\d_{s,\th}^1d_2(\rho_s,\bar\rho_s)ds+\int_0^t\b_{s,\th}^2\D X_s^\th+\d_{s,\th}^2d_2(\rho_s,\bar\rho_s)dB_s^\th\\
        \D Y_t^\th = \b_{T,\th}^3\D X_T^\th+\d_{T,\th}^3d_2(\rho_T,\bar\rho_T)+\int_t^T\b_{s,\th}^4\D X_s^\th+\g_{s,\th}^4\D Z_s^\th+\d_{s,\th}^4d_2(\rho_s,\bar\rho_s)ds-\int_t^T \D Z_s^\th dB_s^\th
    \end{cases}
\end{equation}
where 
$$\b_{s,\th}^1=\frac{\hat b(\th,s,X_s^\th,Z_s^\th,\rho_s)-\hat b(\th,s,\bar X_s^\th,Z_s^\th,\rho_s)}{\D X_s^\th}\dbI_{\{|\D X_s^\th|>0\}},~~\g_{s,\th}^1=\frac{\hat b(\th,s,\bar X_s^\th,Z_s^\th,\rho_s)-\hat b(\th,s,\bar X_s^\th,\bar Z_s^\th,\rho_s)}{\D Z_s^\th}\dbI_{\{|\D Z_s^\th|>0\}}$$
$$\d_{s,\th}^1=\frac{\hat b(\th,s,\bar X_s^\th,\bar Z_s^\th,\rho_s)-\hat b(\th,s,\bar X_s^\th,\bar Z_s^\th,\bar \rho_s)}{d_2(\rho_s,\bar\rho_s)}\dbI_{\{d_2(\rho_s,\bar\rho_s)>0\}}$$
and $(\b_{s,\th}^2,\d_{s,\th}^2)$ for $\si$, $(\b_{T,\th}^3,\d_{T,\th}^3)$ for $G$, and $(\b_{s,\th}^4,\g_{s,\th}^4,\d_{s,\th}^4)$ for $\hat F$ respectively. By the a priori bound of FBSDE, we have:
\begin{equation}
    \begin{aligned}
        &\dbE[\sup_{t\in[0,T]}|\D X_t^\th|^2+\sup_{t\in[0,T]}|\D Y_t^\th|^2+\int_0^T|\D Z_s^\th|^2ds]\\
        \leq&C\dbE[(\int_0^T|\d_{s,\th}^1|d_1(\rho_s,\bar\rho_s)+|\d_{s,\th}^4|d_1(\rho_s,\bar\rho_s)ds)^2+\int_0^T|\d_{s,\th}^2d_1(\rho_s,\bar\rho_s)|^2ds+|\d_{T,\th}^3d_1(\rho_T,\bar\rho_T)|^2]\\
        \leq&CTL^2|d_1(\rho,\bar\rho)|^2
    \end{aligned}
\end{equation}
where $C$ depends on Lipschitz constant $L$. Notice that $|\d_{s,\th}^i|\leq L$ for $i=1,2,3,4$. Thus we have
\begin{equation}
    \begin{aligned}
        d_1(\cL(\bar X),\cL(X))=&\int_I\inf_{\dbP\in\G(\rho_\th,\bar\rho_\th)}\dbE^\dbP\sup_{t\in[0,T]}|X_t^\th-\bar X_t^\th|\l_I(d\th)\\
        \leq&\int_I|\dbE^\dbP\sup_{t\in[0,T]}|X_t^\th-\bar X_t^\th|^2|^\frac{1}{2}\l_I(d\th)\\
        \leq&|CT|^\frac{1}{2}Ld_1(\rho,\bar\rho)
    \end{aligned}
\end{equation}
Thus, for $\d_0\leq \frac{1}{CL^2}$, the mapping is a contraction.
\end{proof}
\begin{thm}
    Let Assumption \ref{c2} hold, and $\d_0$ be the same as in Theorem \ref{wellposed}. For any initial conditions $(t,x,\mu)$, the HMFG \ref{GMFGNE} has a unique value whenever $T\leq\d_0$ .
\end{thm}
\begin{proof}
This is a natural consequence of Theorem \ref{wellposed} and Theorem \ref{fbsdeNE}. Since all the equilibrium $\rho$ can be characterized with strong solutions of the infinite-dimensional FBSDE \eqref{fbsdesystem}, which has a unique solution by Theorem \ref{wellposed}, thus equilibrium $\rho$ is also unique. When $\rho$ is fixed, each player solves a control problem, thus has a unique optimal value which is also the value of HMFG.
\end{proof}

\section{Propagation of Chaos}\label{sect-convergence}
In this section, we show that the equilibrium strategy of the HMFG is a good approximation of the corresponding $N-$player game NE when N is large. 
Furthermore, we need the following assumptions on $b$:
\begin{assum}\label{IContiuity}
    For all $(t,x,\mu,\a)$, the $K$, discretizations of $b$ $\si$, $H$ and $\pa_pH$ converge uniformly in $\dbL^2(\l_I)$ as $K\to\infty$ with rate of convergence $\rho_K$, i.e. $\forall (t,x,\mu,a)\in[0,T]\times\dbR\times\cM_0\times A$
    $$\int_I|b(\th,t,x,\mu,\a)-b(\frac{\lceil{\th K} \rceil }{K},t,x,\mu,\a)|^2d\l_I(d\th)\leq\rho_{K}^2,$$
    and the same holds for $\si$, $H$, and $\pa_pH$.
\end{assum}
\begin{rem}
     Assumption \ref{IContiuity} seems restrictive, but it is consistent with the common assumption in GMFG literature. The propagation of chaos in GMFG literature often assumes conditions that are equivalent to the convergence of truncated step graphon plus the uniform continuity on the graphon aggregate terms. See for example, \cite{aurell2021stochastic, caines2019graphon, cao2025}. The conditions proposed within are similar to Assumption \ref{IContiuity}.  
\end{rem}
Recall that the $N-$player game defined in Definition \ref{nplayer} has $N$ players distributed over a graph $G_K$ with $K$ nodes. The difference between the $N-$player game payoff and HMFG payoff using the control generated from the same control ensemble depends on the features of the network structure of the $N-$player game. Besides the number of players $N$ and the number of nodes $K$, we also need to define the constant:
$$n_{N,K}=\min_{1\leq l\leq K}|\cC_l|.$$
For each $\a\in\cA$, define the corresponding $N-$player control $\a^N\in\cA^N$: 
\begin{equation}\label{nalpha}
    \a_i^N:=\a^{\frac{l}{K}},\q\text{where }i\in \cC_l\q\forall i=1,......,N.
\end{equation}
\begin{lem}\label{propagation}
Let Assumption \ref{c2} and \ref{IContiuity} hold. 
Let $\mu\in\cX_0$, $\bar x\in\dbR^N$ and $\xi\in\dbL_\boxtimes^2(\O\times I;\dbR)$ such that $\xi\sim\mu$ with the following error estimate: denote $\bar x_{l}=(x_{l_1},......,x_{l_{|\cC_l|}})$ where $\{l_i\}_{i=1}^{|\cC_l|}\equiv\cC_l$, and
 $$\d_{N,K}:=\sum_{l=1}^K\frac{1}{K}\Big(W_1^2(\frac{1}{|\cC_l|}\sum_{i\in\cC_l}\d_{ x_i},\mu^{\frac{l}{K}})+\|\bar x_l\|_q^2\Big)+\int_{\frac{l-1}{K}}^{\frac{l}{K}}\dbE|\xi^{ \frac{l}{K}}-\xi^\th|^2\l_I(d\th).$$ 
Suppose $(\a^*,\rho^*)$ is a Nash Equilibrium of HMFG given initial $(t,\mu)$. Define $\a^{*N}$ by using $\a^*$ as in \eqref{nalpha}. Consider the state process $X^{t,\xi,\a^*}$ in HMFG system \ref{fbsdesystem}, and $N-$player dynamics $X^{N,t,\bar x,a^{*N}}$ follows \eqref{npd}, then for any $q>2$, there exists a constant $C$ such that:
\begin{equation}\label{chaos}
\sup_{s\in[t,T]}\dbE d_1^2(\phi_s^N,\cL(X_s^{t,\xi,\a^{*N}}))\leq C\Big(\rho_K^2+\d_{N,K}+\frac{1}{K n_{N,K}^3}\Big).
\end{equation}
 The constant $C$ depends on $T$, $L$, $L_{\a^*}$, and $q$ but not on $N,K$, nor $n_{N,K}$.
    
\end{lem}
\begin{proof}
The approximation is done in two steps. First step, we approximate the difference between the $N-$player game on graph $G_K$ and $K-$node mean field game. Second step, we compare the difference between the $K-$node mean field game with HMFG. For the first step, we will use the classical results on propagation of chaos in MFG theory. The second step needs Assumption \ref{IContiuity}.

We first define the intermediate $K-$node mean field process using Brownian motion $B$ in HMFG:
\begin{equation}\label{intermediate2}
\begin{aligned}
     X_{l,s}^{K,t,\xi,\a^{*N}}=&\xi^{\frac{l}{K}}+\int_t^s b(\frac{l}{K},r, X_{i,r}^{K,t,\xi,\a^{*N}}, \psi_r^{K},a^{*N}_i(r, X_{i,r}^{K,t,\xi,\a^{*N}}))dr\\
    &+\int_t^s \si(\frac{l}{K},r, X_{i,r}^{K,t,\xi,\a^{*N}},\psi_r^{K})dB_r^{\frac{l}{K}}
\end{aligned}
\end{equation}
where $\psi_s^{K}(\th)=\sum_{l=1}^K\dbI_{[\frac{l-1}{K},\frac{l}{K})}(\th)\cL( X_{l,s}^{K,t,\xi,\a^{*N}})$. Different from $\phi_s^N$, $\psi_s^{K}(\th)$ is deterministic. By standard MFG theory,  $X_{l}^{K,t,\xi,\a^{*N}}$ is the limiting process of $ X_{i}^{N,t,\bar x,a^{*N}}$ for $i\in\cC_l$ when $n_{N,K}$ goes to infinity as $N$ goes to infinity while $K$ stays the same. Thus we can apply the standard result in MFG, for example Remark 8.5 (ii) in \cite{mou2019wellposedness}, for any $q>2$, there exists a $C$ such that 
\begin{equation}\label{mfgpc}
    \sup_{s\in[t,T]}\dbE[W_1^2(\frac{1}{|\cC_l|}\sum_{j\in\cC_l}\d_{ X_{s,j}^{N,t,\bar x,a^{*N}}},\cL( X_{s,j}^{K,t,\xi,a^{*N}}))
     ]\leq C W_1^2(\frac{1}{|\cC_l|}\sum_{i\in\cC_l}\d_{x_i},\mu^{l})+\frac{C}{n_{N,K}^3}(1+\|\bar x_l\|_q^q)
\end{equation}
Here $ C$ depends on $T$, $L$ and $q$. Thus the distance between $\psi^{K}$ and $\phi^N$ is
\begin{equation}\label{1st}
    \begin{aligned}
        \sup_{s\in[t,T]}\dbE d_1^2(\psi_s^{K},\phi_s^N)=&\frac{1}{K}\sum_{k=1}^K\sup_{s\in[t,T]}\dbE[W_1^2(\frac{1}{|\cC_k|}\sum_{j\in\cC_k}\d_{X_{s,j}^{N,t,\bar x,a^{*N}}},\cL( X_{s,j}^{K,t,\xi,a^{*N}}))]\\
        \leq&\frac{C}{K}\sum_{k=1}^KW_1^2(\frac{1}{|\cC_k|}\sum_{i\in\cC_k}\d_{ x_i},\mu^{\frac{k}{K}})+\frac{C}{n_{N,K}^3}(1+\|\bar x_k\|_q^q)\leq\d_{N,K}
    \end{aligned}
\end{equation}

We will now analyze the difference between the step law ensemble function of $K-$node mean field process $\psi^K$ and the law ensemble of the HMF process $\rho^*$. The estimates rely on the Assumption \ref{IContiuity}.
\begin{equation}\label{3rd}
    \begin{aligned}
        &d_1^2(\psi_s^{K},\rho_s^*))\\
        \leq&\int_IW_1^2\big(\sum_{l=1}^K \dbI_{[\frac{l-1}{K},\frac{l}{K})}(\th)\cL(X_{l,s}^{K,t,\xi,\a^{*N}}),\rho_s^*) \big)\l_I(d\th)\\
        \leq&C\int_I\dbE|\xi^{\frac{\lceil{\th K} \rceil}{K}}-\xi^\th|^2+\dbE\Big[\int_t^s|\si(\frac{\lceil{\th K} \rceil}{K},r,X_{r,\th}^{t,\xi,\a^*},\psi_r^{K})-\si(\th,r,X_{r,\th}^{t,\xi,\a^*},\rho_r^*)|^2dr\\
        &+|\int_t^s| b(\frac{\lceil{\th K} \rceil}{K},r,X_{r,\th}^{t,\xi,\a^*},\psi_r^{K},\a^{*\frac{\lceil{\th K} \rceil}{K}}(r,X_{r,\th}^{t,\xi,\a^*}))-b(\th,r,X_{r,\th}^{t,\xi,\a^*},\rho_r^*,\a^{*\th}(r,X_{r,\th}^{t,\xi,\a^*}))|dr|^2\Big]\l_I(d\th)\\
        \leq&C\Big(\d_{N,K}+\int_t^sd_1^2(\psi_r^{K},\rho_r^*)dr+\int_I\dbE\Big[\int_t^s|\si(\frac{\lceil{\th K} \rceil}{K},r,X_{r,\th}^{t,\xi,\a^*},\rho_r^*)-\si(\th,r,X_{r,\th}^{t,\xi,\a^*},\rho_r^*)|^2dr\\
        &+\int_t^s|\pa_p H(\frac{\lceil{\th K} \rceil}{K},r,X_{r,\th}^{t,x,\a^*},\pa_xu_{\frac{\lceil{\th K} \rceil}{K}}^{\rho^*}(r,X_{r,\th}^{t,x,\a^*}),\rho_r^*)-\pa_p H(\th,r,X_{r,\th}^{t,x,\a^*},\pa_xu_{\th}^{\rho^*}(r,X_{r,\th}^{t,x,\a^*}),\rho_r^*)|^2dr\Big]\l_I(d\th)
    \end{aligned}
\end{equation}
Since $\pa_pH$ is uniformly Lipschitz in $p$ by Assumption \ref{c2} and $\pa_pH(\frac{\lceil{\th K} \rceil}{K},\cdot)$ and $\pa_pH(\th,\cdot)$ are close by Assumption \ref{IContiuity}, it suffices to check the difference between $\pa_xu_{\frac{\lceil{\th K} \rceil}{K}}^{\rho^*}$ and $\pa_xu_{\th}^{\rho^*}$. By equation (2.46) in \cite{ZhangThesis}: we have:
\begin{equation*}
\begin{aligned}
    &\int_I|\pa_xu_{\frac{\lceil{\th K} \rceil}{K}}^{\rho^*}(t,x)-\pa_xu_{\th}^{\rho^*}(t,x)|\l_I(d\th)\\
    \leq &\int_I(1+\sqrt{T-t})\Big(\|G(\frac{\lceil{\th K} \rceil}{K},\cdot)-G(\th,\cdot)\|_\infty+\|\si(\frac{\lceil{\th K} \rceil}{K},\cdot)-\si(\th,\cdot)\|_\infty\\
    &+\|H(\frac{\lceil{\th K} \rceil}{K},\cdot)-H(\th,\cdot)\|_\infty\Big)\l_I(d\th)\\
    \leq& C\rho_K
\end{aligned}
\end{equation*}
Thus by Gronwall's inequality, we have
\begin{equation}\label{2nd}
    d_1^2(\psi_s^{K},\rho_s^*)\leq C(\d_{N,K}+\rho_K+\int_t^sd_1^2(\psi_r^{K},\rho_r^*)dr)\leq C(\d_{N,K}+\rho_K)
\end{equation}
Combining \eqref{1st} and \eqref{2nd}, we get the desired result.
\end{proof}
\begin{thm}
     Let Assumption \ref{c2} and \ref{IContiuity} hold. Given initial time $t$, let $\mu\in\cX_0$, $\bar x\in\dbR^N$ and $\xi\in\dbL_\boxtimes^2(\O\times I;\dbR)$ such that $\xi\sim\mu$ be the initial conditions with the same error estimate as in Lemma \ref{propagation}. Suppose $(\a^*,\rho^*)$ is a Nash Equilibrium of HMFG given initial $(t,\mu)$. Define $\a^{*,N}$ by using $\a^*$ as in \eqref{nalpha}.  Then there exists an $\e_{N,K}$ such that $\lim_{N,K\to\infty}\e_{N,K}=0$ and:
    $$J^N(i,t,\bar x,\a^{*,N})\geq J^N(i,t,\bar x,\a^i,\a^{*,N,-i})-\e_{N,K}\quad\forall\a^{i}\in\cA^i\quad\forall i=1,......,N.$$
\end{thm}
\begin{proof}
    For simplicity, we show for a fixed player $i$ in node $l_i$ with fixed initial time $0$.
    
    Step 1. We show that, in $N$-player game when only a single player change her strategy, the law ensemble of all players will only deviate in a scale of $\frac{1}{\sqrt{N}}$. Fix arbitrary $\a^i\in\cA^i$. Let us denote the $N$-player dynamic $(X_{s,1}^{N,\a^{*,-i}},......,X_{s,n}^{N,\a^{*,-i}})$ as each player $j\neq i$ uses strategy $\a^{*\frac{l}{K}}$ where $j\in\cC_l$ and player $i$ uses strategy $\a^i$.
    
    Let us consider another $(N-1)$-dimensional processes with all players expect $i$, and $\forall j\in\{1,......,N\} $, $j\neq i $, player $j$ use strategy $\a^{*\frac{l}{K}}$ where $j\in\cC_l$. The dynamic of the player $j$ is the following
\begin{equation}
    \tl X_{t,j}^{N,\a^{*}}=x_j+\int_0^t b(\frac{l}{K},r,\tl X_{s,j}^{N,\a^*},\tl\phi_s^{N},\a_{s}^{*,\frac{l}{K}})ds+\int_0^t \si(\frac{l}{K},s,\tl X_{s,j}^{N,\a^*},\tl\phi_s^{N})dB_s^j
\end{equation}
where $\a_{s}^{*,\frac{l}{K}}:=\a^{*,\frac{l}{K}}(s,\tl X_{s,i}^{N,\a^{*}})$ for simplicity and for each $\th\in I$,
$$\tl\phi_s^{N}(\th)=\frac{1}{|\cC_{\left \lceil{\th K}\right \rceil }|-\dbI_{\{i\in\cC_{\left \lceil{\th K}\right \rceil }\}}}\sum_{j\in\cC_{\left \lceil{\th K}\right \rceil },j\neq i}\d_{\tl X_{s,j}^{N,\a^*}}.$$
Then we compare $\tl\phi_s^{N}$ with $\phi_s^{N}$:
\begin{equation}
    \begin{aligned}
        \dbE d_1^2(\tl\phi_s^{N},\phi_s^{N})=&\frac{1}{K}\dbE\Big[\sum_{k=1}^KW_1(\frac{1}{|\cC_k|-\dbI_{\{i\in\cC_k\}}}\sum_{j\in\cC_k,j\neq i}\d_{\tl X_{s,j}^{N,\a^*}},\frac{1}{|\cC_k|}\sum_{j\in\cC_k}\d_{X_{s,j}^{N,\a^{*,-i}}})\Big]^2\\
        \leq&\frac{C}{K}\dbE\Big[\sum_{k\neq l_i}\frac{1}{\cC_k}\sum_{j\in\cC_k}|\tl X_{s,j}^{N,\a^*}-X_{s,j}^{N,\a^{*,-i}}|^2+\frac{1}{\cC_{l_i}-1}\sum_{j\in\cC_{l_i},j\neq_i}|\tl X_{s,j}^{N,\a^*}-X_{s,j}^{N,\a^{*,-i}}|^2\\
        &+\frac{1}{\cC_{l_i}(\cC_{l_i}-1)}\sum_{j\in\cC_{l_i},j\neq_i}|X_{s,i}^{N,\a^{*,-i}}-X_{s,j}^{N,\a^{*,-i}}|^2\Big]\\
        \leq&\frac{C}{K}\int_0^s\dbE d_2^2(\tl\phi_r^{N},\phi_r^{N})dr +\frac{C}{|C_{l_i}|}\\
        \leq&\frac{C}{Kn_{N,K}}
    \end{aligned}
\end{equation}
The exact value of constant $C$ changes between lines and depends only on $L$, $T$ and $L_{\a^*}$. The first inequality follows from the fact that for any empirical measure on $\cP(\dbR)$, 
$$W_2^2(\frac{1}{n-1}\sum_{i=1}^{n-1}\d_{x_i},\frac{1}{n}\sum_{i=1}^{n-1}\d_{x_i})\leq\frac{1}{N(N-1)}\sum_{i=1}^{n-1}|x_i-x_n|^2$$
The second inequality holds because all the state process are in $\dbS^2$ under the running Assumption \ref{lipschitz} and standard comparison principle is SDE theory. The last inequality follows from Gronwell's inequality.

Furthermore, by Theorem \ref{propagation}, we have
$$\sup_{t\in[0,T]}\dbE d_1^2(\phi^N_t,\rho_t^*)\leq\sup_{t\in[0,T]}\dbE d_2(\tl\phi^N_t,\rho_t^*)+\sup_{t\in[0,T]}\dbE d_2(\tl\phi^N_t,\phi^N_t)\leq C\Big(\rho_{K}+\d_{N,K}+\frac{1}{K( n_{N,K}-1)^3}\Big)$$
Thus it suffices to use $\rho^*$ instead of $\phi^N$ in player $i-th$ optimization problem in $N$-player games since $\phi^N$ will perturb if player $i$ use different control $\a^i$ while $\rho^*$ will not. 
    
    Step 2, Consider the optimal control problem
    \begin{equation}\label{ioptimal}
    \begin{aligned}
        V_i(0, x_i;\rho^*)=\sup_{\a^i\in\cA^i}\dbE[G(\frac{l_i}{k},\tl X_{T,i}^{N,\rho^*,\a^i},\rho^*_T)+\int_0^T F(\frac{l_i}{k},s,\tl X_{s,i}^{N,\rho^*,\a^i},\rho_s^*,\a_s^{i})ds]\\    
        \tl X_{t,i}^{N,\rho^*,\a^i}=x_i+\int_0^s b(\frac{l_i}{K},s,\tl X_{s,i}^{N,\rho^*,\a^i},\rho^*_s,\a_{s}^{i})ds+\int_0^t \si(\frac{l_i}{K},s,\tl X_{s,u}^{N,\rho^*,\a^i},\rho_s^*)dB_s^i
    \end{aligned}
    \end{equation}
    Since $(\a^*,\rho^*)$ is a NE of HMFG, then for fixed population $\rho^*$, $\a^{*l_i}$ solves the optimization problem of type $l_i$ player, which is exactly the problem described in \eqref{ioptimal}.

    Now let us compare optimization problem \eqref{ioptimal} and the game problem for player $i$. First compare the dynamics $\tl X_{t,i}^{N,\rho^*,\a^{*l_i}}$ and $X_{s,i}^{N,\a^*}$:
    \begin{equation}
        \begin{aligned}
            \dbE\sup_{[0,T]}|\tl X_{t,i}^{N,\rho^*,\a^{*,l_i}}-X_{s,i}^{N,\a^*}|^2\leq\dbE C\int_0^Td_1^2(\rho^*_t,\phi_t^N)dt\leq C\Big(\rho_{K}+\d_{N,K}+\frac{1}{K( n_{N,K}-1)^3}\Big)
        \end{aligned}
    \end{equation}
   By Assumption \ref{costassum}, the difference in payoff is:
    \begin{equation*}
        \begin{aligned}
            &|J^N(i,0,\bar x,\a^{*N})-J(l_i,0,\bar x_i,\a^{*,l_i};\rho^*)|^2\\
            \leq&C\dbE\Big[|\tl X_{T,i}^{N,\rho^*,\a^{*,l_i}}-X_{T,i}^{N,\a^*}|^2+d_1^2(\phi_T^N,\rho_T^*)+\int_0^T |\tl X_{t,i}^{N,\rho^*,\a^{*,l_i}}-X_{t,i}^{N,\a^*}|^2+d_1^2(\phi_t^N,\rho_t^*)dt\Big]\\
            \leq&C\Big(\rho_{K}+\d_{N,K}+\frac{1}{K( n_{N,K}-1)^3}\Big)
        \end{aligned}
    \end{equation*}
    Again, $C$ is a constant that changes between lines but depends only on $T$, $L$, $L_{a^*}$ and $K$. Thus it suffices to take $\e_{N,K}=C\Big(\rho_{K}+\d_{N,K}+\frac{1}{K( n_{N,K}-1)^3}\Big)^\frac{1}{2}$

\end{proof}

\section{Master Equation}\label{sect-MasterEq}
The goal of this section is to develop the so-called master equation for HMFG. It is a PDE that represents the dynamics of the agent's payoffs. In this project, we only derive the PDE as the decoupling field of the infinite-dimensional FBSDE system \eqref{fbsdesystem} using the It\^o formula in Theorem \ref{ito} on the marginal law ensemble. The global well-posedness of the master equation is a serious issue and will be investigated in our future work. First we will derive the flow property which resembles the MFG theory.
By Theorem \ref{wellposed}, we have $V(\th,t,x,\mu)=Y_{t,\th}^{t,x,\mu}$ where $(X^{t,\xi,\mu},Y^{t,\xi,\mu},Z^{t,\xi,\mu})$ solves \eqref{fbsdesystem}. Denote  $\rho_s=\cL(X_{s}^{t,\xi,\mu})$ for $s\in[t,T]$. By \eqref{zNu} in the proof of Theorem \ref{wellposed}, we know 
\begin{equation}\label{z2}
    Z_{s,\th}^{t,\xi,\mu}=\si(\th,s,X_{s,\th}^{t,\xi,\mu},\rho_s)\pa_xu_\th^{\rho}(s,X_{s,\th}^{t,\xi,\mu})
\end{equation}
Furthermore, suppose $V$ satisfies $Y_{s,\th}^{t,\xi,\mu}=V(\th,s,X_{s,\th}^{t,x,\mu},\rho_s)$, we have for all $s\in[t,T]$
\begin{equation}\label{y2}
    Y_{s,\th}^{t,\xi,\mu}=Y_{s,\th}^{s,X_s^{t,\xi,\mu},\rho_s}=V(\th,s,X_{s,\th}^{t,x,\mu},\rho_s)=u_\th^{\rho}(s,X_s^{t,x,\mu})
\end{equation}
thus we further have $\pa_xV(\th,s,X_{s,\th}^{t,x,\mu},\rho_s)=\pa_xu_\th^{\rho}(s,X_s^{t,x,\mu})$. 
Plug \eqref{z2}-\eqref{y2} into \eqref{fbsdesystem}, then $(X^{t,\xi,\mu},Y^{t,\xi,\mu},Z^{t,\xi,\mu})$ also solves the following decoupled FBSDE:
\begin{equation}\label{decoupledfbsde}
    \begin{cases}
         X_{s,\th}^{t,\xi,\mu}=&\xi^\th+\int_t^s  \pa_p H(\th,r,X_{r,\th}^{t,\xi,\mu},\pa_xV(\th,r,X_{r,\th}^{t,\xi,\mu},\rho_r)dr + \int_t^s \si(\th,r,X_{r,\th}^{t,\xi,\mu},\rho_r)dB_r^\th\\
         Y_{s,\th}^{t,\xi,\mu}=&G(\th,X_{T,\th}^{t,\xi,\mu},\rho_T)+\int_s^T H(\th,r,X_{r,\th}^{t,\xi,\mu},\pa_xV(\th,r,X_{r,\th}^{t,\xi,\mu},\rho_r),\rho_r)\\
         &-\pa_p H(\th,r,X_{r,\th}^{t,\xi,\mu},\pa_xV(\th,r,X_{r,\th}^{t,\xi,\mu},\rho_r),\rho_r)\pa_xV(\th,r,X_{r,\th}^{t,\xi,\mu},\rho_r)dr-\int_s^T Z_{r,\th}^{t,\xi,\mu} dB_r^\th
    \end{cases}
\end{equation}
Apply the It\^o Formula derived in Theorem \ref{ito} on $V(s,X_{s,\th}^{t,x,\mu},\rho_s)$, we have 
\begin{equation}\label{master1}
    \begin{aligned}
        V(\th,s,X_{s,\th}^{t,\xi,\mu},\rho_s)=&V(\th,t,\xi^\th,\mu)+\int_t^s\Big[\pa_t V(\th,r,X_{r,\th}^{t,\xi,\mu},\rho_r)\\
        &+\pa_x V(\th,r,X_{r,\th}^{t,\xi,\mu},\rho_r)\pa_pH (\th,r,X_{r,\th}^{t,\xi,\mu},\partial_x u^\rho_\th(r,X_{r,\th}^{t,\xi,\mu}),\rho_r)\\
        &+\frac{1}{2}\pa_{xx}V(\th,r,X_{r,\th}^{t,\xi,\mu},\rho_r)|\si(\th,r,X_{r,\th}^{t,\xi,\mu},\rho_r)|^2\\
        &+\int_{I}\tl \dbE\Big(\partial_{\tl x}\frac{\d V}{\d \mu}(\th,r,X_{r,\th}^{t,\xi,\mu},\rho_r,\tl X_{r,\tl\th}^{t,\tl \xi,\mu},\tl\th)\Big) \partial_p H(\tl\th,s,\tl X_{r,\tl \th}^{t,\tl\xi,\mu},\partial_x u^\rho_{\tl\th}(s,\tl X_{r,\tl \th}^{t,\tl\xi,\mu},\rho_r)\\
        &+\frac{1}{2}\partial_{\tl x\tl x}\frac{\d V}{\d \mu}(\th,r,X_{r,\th}^{t,\xi,\mu},\tl X_{r,\tl\th}^{t,\tl\xi,\mu}\rho_r,\tl\th)|\si(\tl\th,r,\tl X_{r,\tl\th}^{t,\tl\xi,\mu},\rho_r)|^2 \Big)\l_I(d\tl\th )\Big]dr
    \end{aligned}
\end{equation}
Note that $\tl X_{r,\tl \th}^{t,\xi,\mu}$ has the same distribution as $ X_{r,\tl \th}^{t,\tl\xi,\mu}$ but is independent with $\tl X_{r,\tl \th}^{t,\xi,\mu}$ and $\tl\dbE$ means taking the expectation with respect to $\tl X_{r,\tl \th}^{t,\xi,\mu}$ only.

Combine the backward equation in \eqref{decoupledfbsde} and equation \eqref{master1}, we have master equation of HMFG:
\begin{equation}\label{master}
    \begin{aligned}
        0=&\pa_t V(\th,t,x,\rho)+H(\th,t,x,\pa_xV,\rho)+\frac{1}{2}\pa_{xx}V(\th,t,x,\rho)|\si(\th,t,x,\rho)|^2\\
        &+\int_{I} \dbE\Big(\partial_x\frac{\d V}{\d \mu}(\th,t,x,\rho,\xi^{\tl\th},\tl\th)\Big) \partial_y H(\tl\th,t, \xi^{\tl\th},\partial_x V(\tl\th,t, \xi^{\tl\th},\rho),\rho)\\
        &+\frac{1}{2}\partial_{xx}\frac{\d V}{\d \mu}(\th,t,x,\xi^{\tl\th},\rho,\tl\th)|\si(\tl\th,t,\xi^{\tl\th},\rho)|^2 \Big)\l_I(d\tl\th ),\\
        V(\th,T,&x,\rho)=G(\th,x,\rho),\quad\forall (\th,t,x,\rho)\in I\times[0,T]\times\dbR\times\cM_0, \text{ where } \cL(\xi)=\rho
    \end{aligned}
\end{equation}
\begin{defn}
    Denote $V$ is a classical solution of \eqref{master} if \\
    (i) $\forall\th\in I$, $\pa_tV(\th,\cdot)$, $\pa_xV(\th,\cdot)$, and $\pa_{xx}V(\th,\cdot)$ exist and are continuous in $(t,x,\rho)$ ;\\
    (ii) $\forall\th,\tl\th\in I$$\pa_{\tl x}\frac{\d V}{\d \mu}(\th,\tl\th,\cdot)$ and $\pa_{\tl x\tl x}\frac{\d V}{\d \mu}(\th,\tl\th,\cdot)$ exist and are continuous in $(t,x,\rho,\tl x)$;\\
    (iii) V solves \eqref{master}.
\end{defn}
Then we have the following theorem.
\begin{thm}
    If the master equation \eqref{master} has a classical solution $V$, then the FBSDE system \eqref{fbsdesystem} has a unique solution and $V$ is a decoupling field of \eqref{fbsdesystem}, i.e. 
    \begin{equation}\label{decouplingfield}
        (Y_{s,\th}^{t,\xi}, Z_{s,\th}^{t,\xi})=(V(\th,s, X_{s,\th}^{t,\xi},\rho_s),\si(\th,s,X_{s,\th}^{t,\xi},\rho_s)\pa_xV(\th,s,X_{s,\th}^{t,\xi},\rho_s))
    \end{equation}
\end{thm}
\begin{proof}
    First show the existence of by checking \eqref{decouplingfield} provide a solution of the FBSDE system \eqref{fbsdesystem}
Since $V$ is a classical solution of the PDE \eqref{master}, the following decoupled FBSDE system is well-posed:
\begin{equation}\label{decoupledfbsdeV}
    \begin{cases}
         \tl X_{s,\th}^{t,\xi}=\xi^\th+\int_t^s  \pa_p H(\th,r,\tl X_{r,\th}^{t,\xi},\pa_xV(\th,r,\tl X_{r,\th}^{t,\xi},\tl\rho_r)dr + \int_t^s \si(\th,r,\tl X_{r,\th}^{t,\xi},\tl\rho_r)dB_r^\th\\
         \tl Y_{s,\th}^{t,\xi}=V(\th,t_1,\tl X_{t_1,\th}^{t,\xi},\tl\rho_{t_1})+\int_s^{t_1} H(\th,r,\tl X_{r,\th}^{t,\xi},\pa_xV(\th,r,\tl X_{r,\th}^{t,\xi},\tl\rho_r),\tl\rho_r)\\
         ~~~~~~~~~~~-\pa_p H(\th,r,\tl X_{r,\th}^{t,\xi},\pa_xV(\th,r,\tl X_{r,\th}^{t,\xi},\tl\rho_r),\tl\rho_r)\pa_xV(\th,r,\tl X_{r,\th}^{t,\xi},\tl\rho_r)dr-\int_s^{t_1} \tl Z_{r,\th}^{t,\xi} dB_r^\th\\
         ~\forall s\in[t,t_1]\quad\forall\th\in I,
    \end{cases}
\end{equation}
 then $(\tl Y_{s,\th}^{t,\xi},\tl Z_{s,\th}^{t,\xi})=(V(\th,s,X_{s,\th}^{t,\xi},\tl\rho_s),\si(\th,s,\tl X_{s,\th}^{t,\xi},\rho_s)\pa_xV(\th,s,X_{s,\th}^{t,\xi},\tl\rho_s))$, where $\tl\rho_s=\cL(\tl X_{s}^{t,\xi})$. By \eqref{master1}, $(\tl X^{t,\xi},\tl Y^{t,\xi},\tl Z^{t,\xi})$ is a solution of \eqref{fbsdesystem}.

For uniqueness, let $(X,Y,Z)\in\dbL_\boxtimes^2$ by an arbitrary solution of \eqref{fbsdesystem} and let us consider the solution of FBSDE system \eqref{fbsdesystem} on time interval $[t_n,T]$ where $T-t_n\leq\d_0$ defined in Theorem \ref{wellposed} given the initial $(t_n,X_{t_n},\cL(X_{t_n}))$. Then we have that given initial state ensemble $X_{t_n}$, the solution is unique and satisfies \eqref{decouplingfield} on time interval $[t_n,T]$. Thus $Y_{t_n,\th}^{t_n,X_{t_n,\th},\cL(X_{t_n})}=V(t_n,\th,X_{t_n,\th},\cL(X_{t_n}))$ and $(X^{t_n,X_{t_n,\th},\cL(X_{t_n})},Y^{t_n,X_{t_n,\th},\cL(X_{t_n})},Z^{t_n,X_{t_n,\th},\cL(X_{t_n})})=(X,Y,Z)$. 
Then we define $0=t_0<t_1<......<t_n<T$ and show by backward induction. On time interval $[t_{i-1},t_i]$, we solve the FBSDE system
\begin{equation}\label{fbsdesystemti}
    \begin{cases}
        X_{s,\th}^{t_{i-1},X_{t_{i-1}},\cL(X_{t_{i-1}})}=&X_{t_i}+\int_{t_i}^s \hat b(\th,r,X_{r,\th}^{t_{i-1},X_{t_{i-1}},\cL(X_{t_{i-1}})},Z_{r,\th}^{t_{i-1},X_{t_{i-1}},\cL(X_{t_{n-1}})},\rho_r^i)dr \\
        &+ \int_t^s \si(\th,r,X_{r,\th}^{t_{n-1},X_{t_{n-1}},\cL(X_{t_{n-1}}},\cL(X_{r}^{t,\xi,\mu}))dB_r^\th\\
        Y_{s,\th}^{t_{i-1},X_{t_{i-1}},\cL(X_{t_{i-1}})}=&V(\th,t_i,X_{i,\th},\cL(X_{i}))-\int_s^T Z_{r,\th}^{t_{i-1},X_{t_{i-1}},\cL(X_{t_{i-1}})} dB_r^\th\\
        &+\int_s^{t_i} \hat F(\th,r,X_{r,\th}^{t_{i-1},X_{t_{i-1}},\cL(X_{t_{i-1}})},Z_{r,\th}^{t_{i-1},X_{t_{i-1}},\cL(X_{t_{i-1}})},\rho_r^i)dr
    \end{cases}
\end{equation}
where $\rho_r^i=\cL(X_{r}^{t_{i-1},X_{t_{i-1}},\cL(X_{t_{i-1}})})$. By Theorem \ref{wellposed}, we get $Y_{t_{i-1},\th}^{t_i,X_{t_i,\th},\cL(X_{t_i})}=V(t_i,\th,X_{t_i,\th},\cL(X_{t_i}))$ and $(X^{t_i,X_{t_i,\th},\cL(X_{t_i})},Y^{t_i,X_{t_i,\th},\cL(X_{t_i})},Z^{t_i,X_{t_i,\th},\cL(X_{t_i})})=(X,Y,Z)$. Thus the solution of $\eqref{fbsdesystem}$ is indeed unique on $[0,T]$.
\end{proof}

\bibliographystyle{plain}
\bibliography{refs}

\end{document}